\documentclass{ws-ijm}
\def\<{\langle}
\def\>{\rangle}
\def\a{\alpha}

\def\c{\cdot}

\def\D{\Delta}

\def\g{\gamma}

\def\om{\omega}

\def\i{\iota}
\def\l{\lambda}

\def\lr{\longrightarrow}
\def\rh{\rightharpoonup }

\def\o{\otimes}
\def\p{\phi}

\def\v{\varepsilon}
\def\vp{\varphi}

\def\<{\langle}
\def\>{\rangle}

\begin{document}

\markboth{N. Zhou, S. Wang }
{A Duality Theorem for Weak Multiplier Hopf Algebra Actions}


\title{A DUALITY THEOREM FOR WEAK MULTIPLIER HOPF ALGEBRA ACTIONS }

\author{Nan Zhou, Shuanhong Wang\footnote{Corresponding author}}
\address{Department of Mathematics, Southeast University \\
Nanjing, Jiangsu, China\\
zhounan0805@163.com, shuanhwang@seu.edu.cn}

\maketitle

\begin{abstract}
The main purpose  of this paper is to unify  the theory of actions of Hopf algebras, weak Hopf algebras and multiplier Hopf algebras to one of actions of weak multiplier Hopf algebras introduced by A. Van Daele and S. H. Wang. Using such developed actions, we will define the notion of a module algebra over weak multiplier Hopf  algebras  and  construct their smash products. The main result is the duality theorem for actions  and their dual actions on the smash product of  weak multiplier Hopf algebras. As an application, we recover the main results found in the literature for weak Hopf algebras, multiplier Hopf algebras and groupoids.
\end{abstract}



\section{Introduction}	
It is well known that the most well-known examples of Hopf algebras (\cite{S}) are the linear spans of (arbitrary)
 groups. Dually, also the vector space of linear functionals on a finite group carries the
 structure of a Hopf algebra. In the case of infinite groups, however, the vector space
 of linear functionals (with finite support)  possesses no unit. Consequently, it is no
 longer a Hopf algebra but, more generally, a multiplier Hopf algebra in \cite{VD1} and \cite{VD2}.
  Considering finite groupoids, both their linear spans and the dual vector spaces of linear
 functionals carry weak Hopf algebra structures in \cite{BNS} and \cite{BS}. Finally, removing the finiteness
 constraint in this situation, both the linear spans of arbitrary groupoids and the vector
 spaces of linear functionals with finite support on them are examples of weak multiplier
 Hopf algebras as introduced in \cite{VDW1} and \cite{VDW2}
 (also see \cite{B1} \cite{B2} \cite{BGL} and \cite{KV} for more references). Especially,
 multiplier Hopf algebras and weak Hopf algebras are the special examples of weak multiplier Hopf algebras.

 In the classical Hopf algebra action theory: Let $H$ be a finite Hopf algebra and $A$ be a left $H$-module algebra. Then
 the smash product $A\# H$ is a left $H^*$-module algebra via a natural way: $h^*\c (a\#h)=(a\# h^*\rh h)$,
 where $h^*\rh h=\sum h_{(1)}\<h^*, h_{(2)}\>$
 for all $a\in A, h\in H, h^*\in H^*$.
 Furthermore, it follows from \cite{BM} that there is an isomorphism
 $(A\# H)\# H^* \cong M_n(A)$, where $n = dim H$ and $M_n(A)$ is an algebra of $n$-by-$n$ matrices over $A$.
  For weak  Hopf algebras this result was proved in \cite{N} (also see \cite{ZW}). For groupoids the result was studied in \cite{PF}.
 For multiplier Hopf algebras with integrals  it was established in \cite{DVZ} (where infinite dimensional case was considered).
 In this paper, we will show that this result can be extended to weak multiplier Hopf algebras with integrals
 in some form.

The paper is organized as follows.

First we recall some definitions and propositions related to weak multiplier Hopf algebras in Section \ref{two}. In Section \ref{three} we define the
 notion of an action of a weak multiplier Hopf algebra $A$ on an algebra $R$  and study some examples.   In Section \ref{four},  we discuss  the smash product $R\#A$ for a left $A$-module algebra $R$.

In Section \ref{five} we consider the pairing between a weak multiplier Hopf algebra $A$ with integrals
  and its dual $\hat{A}$.
 The main result is Proposition \ref{5.5} and \ref{5.6}. These results will be basic for the theory in the next section.

 In Section \ref{six} we will get the main results of this paper. We first get some properties of bi-smash products.
 Then we apply them to any dual pair of algebraic quantum groupoids. Finally we obtain the main duality theorem
 in Theorem \ref{6.4}.

In the paper, we always consider the non-unital associative algebra $A$ over $\mathbb{C}$ with non-degenerate product.
 By this one means that: given $a\in A$ we have $a = 0$ if either $ab = 0$ for all $b\in A$ or $ba = 0$ for all $b\in A$.
 We will also require the algebra $A$ to be idempotent (i.e. $A^2=A$). Clearly, also the algebra $A^{op}$ on the same vector space
 $A$ with the opposite multiplication is idempotent and non-degenerate, whenever $A$ is so.

 For any algebra $A$,  recall from  \cite{VD1} that a left multiplier of $A$ is a linear map $l: A \lr A$ such that
 $l(ab) = l(a)b$ for all $a, b \in A$. A right multiplier of $A$ is a linear map
 $r: A \lr A$ such that $r(ab) = ar(b)$ for all $a, b \in A$. A multiplier of $A$ is a
 pair $(l, r)$ of a left and a right multiplier such that $r(a)b = al(b)$ for all
 $a, b \in A$. We denote by $L(A)$, $R(A)$, and $M(A)$ the left, right, and multipliers of $A$.
 It is clear that the composition of maps makes these vector spaces into algebras.
  If $A$ has an identity then $A=M(A)=L(A)=R(A)$ (cf.\cite{VD1}).

  We will use $1$ for the identity in $M(A)$  and we will use $\i $ to denote
  the identity map of each vector space.
  And if $B$ is another non-unital non-degenerate algebra, we have the following natural embeddings
 $A\o B\subseteq M(A)\o M(B)\subseteq M(A\o B)$.

\section{Preliminaries on  weak multiplier Hopf algebras}\label{two}
In this section  we will recall the definition and some properties of  weak multiplier Hopf algebras.

Recall from the Definition 1.1 in \cite{VDW1} that
 a {\it coproduct} on an  algebra $A$ is a homomorphism $\D: A\lr M(A\o A)$ such that
\begin{itemlist}
 \item $\D(a)(1\o b)$ and $(a\o 1)\D(b)$ are in $A\o A$ for all $a, b\in A$,
 \item $\D$ is coassociative in the sense that
$$
(c\o 1\o 1)(\Delta\o \i)(\Delta(a)(1\o b))=(\i\o\Delta)((c\o 1)\Delta(a))(1\o 1\o b)
$$
for all $a, b, c\in A$.
\end{itemlist}

The coproduct $\D$ is called \textit{full} if the smallest subspaces $V$ and $W$ of $A$ satisfying
$$\D(A)(1\o A)\subseteq V\o A \quad \quad \text{and} \quad \quad (A\o 1)\D(A)\subseteq A\o W$$
are equal to $A$. The coproduct $\D $ is called \textit{regular} if $\D(a)(b\o 1)$ and $(1\o a)\D(b)$ are in $A\o A$ for all $a, b\in A$.

For a coproduct $\D$ on $A$, the \textit{canonical maps} $T_1$ and $T_2$ from $A\o A$ to $A\o A$ are defined by
$$
T_1(a\o b)=\D(a)(1\o b) \quad \text{and} \quad T_2(a\o b)= (a\o 1)\D(b).
$$
If $\D$ is regular, we define $T_3$ and $T_4$ on $A\o A$ by
$$
T_3(a\o b)=(1\o b)\D(a) \quad \text{and} \quad T_4(a\o b)=\D(b)(a\o 1).
$$

Recall from the  Definition 1.8 of \cite{VDW1} that a linear map $\v: A\lr \mathbb{C}$ is called a {\it counit}
on an algebra  $A$ with a coproduct $\D $ if
$$
(\v\o \i)(\D(a)(1\o b))=ab \quad \text{and} \quad (\i\o \v)((a\o 1)\D(b))=ab
$$
for all $a, b\in A$.

For any vector space $A$, denote by ${\bf fl}$ the flip map $A\o A \lr  A\o A, a\o b \mapsto b\o a$.
For a non-unital algebra $A$ with a non-degenerate multiplication, and for a coproduct $\D: A\lr M(A\o A)$,
define a multiplicative linear map $\D ^{cop}: A\lr M(A\o A)$
via
$$
\D ^{cop}(a)(b \o c):= {\bf fl}(\D (a)(c \o b))\, \mbox { and}\,  (b \o c)\D ^{cop}(a):= {\bf fl}((c \o b)\D (a))
$$
and define $ \D _{13}: A \lr  M(A \o A \o A)$ by
$$
\D _{13}(a)(b \o c \o d):= (\i  \o {\bf fl})(\D (a)(b \o d) \o c)
$$
 and
$$
(b \o c \o d)\D _{13}(a):= (\i \o {\bf fl})((b \o d)\D (a) \o c).
$$

Recall from the Definition 1.14 of \cite{VDW2}  that
 a {\it weak multiplier Hopf algebra} is a pair $(A,\D)$ of a non-degenerate idempotent algebra $A$ with a full coproduct
 $\D $ and a counit satisfying the following conditions:

\begin{romanlist}[(iii)]
\item there exists an idempotent $E\in M(A\otimes A)$ giving the ranges of the canonical maps:
$$
E(A\otimes A)=T_1(A\otimes A)\quad \text{and}\quad (A\otimes A)E=T_2(A\otimes A);
$$
\item the element $E$ satisfies
$$
(\i\otimes \D)E=(E\otimes 1)(1\otimes E)=(1\otimes E)(E\otimes 1);\eqno(2.1)
$$
\item  the kernels of the canonical maps are of the form
$$
Ker(T_1)=(1-G_1)(A\otimes A)\quad \text{and}\quad Ker(T_2)=(1-G_2)(A\otimes A)
$$
where $G_1$ and $G_2$ are the linear maps from $A\otimes A$ to itself, given as
$$
(G_1\otimes \i)(\D_{13}(a)(1\otimes b\otimes c))=\D_{13}(a)(1\otimes E)(1\otimes b\otimes c)
$$
and
$$
(\i\otimes G_2)((a\otimes b\otimes 1)\D_{13}(c))=(a\otimes b\otimes 1)(E\otimes 1)\D_{13}(c)
$$
for all $a, b, c\in A$.
\end{romanlist}

A weak multiplier Hopf algebra is called \textit{regular} if the coproduct is regular and if
also $(A, \D^{cop})$ is a weak  multiplier Hopf algebra. This is the same as requiring that also $(A^{op}, \D)$ is
 a weak  multiplier Hopf algebra.

\begin{remark}
i) Let $(A, \D )$ be a regular weak multiplier Hopf algebra. As the same as in a regular multiplier Hopf algebra case (see Proposition 2.2 in \cite{DVZ}),
 we know from Proposition 2.21 in \cite{VDW3} that $A$ has local units. More precisely, let $a_1, a_2, \cdots , a_n$ be elements in $A$. Then there exist elements
 $e, f$ in $A$ such that $ea_i=a_i$, and $a_if=a_i$, for all $i$.

We would like to use a formal expression for $\D (a)$ when $a \in A$. The problem is
 that $\D (a)$ is not in $A \o A$ in general. We know however that
 $\D(a)(1\o b)\in A\o A$ for all $a, b\in A$.  We will use the Sweedler notation and
 write $\D(a)=\sum a_{(1)}\o a_{(2)}\in M(A\o A)$ for any $a\in A$.
 Then we know that $\D(a)(1\o b)\in A\o A$ for all $a,b\in A$ and  we will write
$\D(a)(1\o b)=\sum a_{(1)}\o a_{(2)}b$.
Then we say that $a_{(2)}$ is {\it covered} by $b$ in this equation.
Now, we know that there is an element $e \in A$
  such that $b=eb$ and we can think
 of $\sum a_{(1)}\o a_{(2)}$ to stand for  $\D(a)(1\o e)$. Of course, this is still dependent on $b$.
 But we know that for several elements $b$, we can use the same $e$.
   Note that it should be careful when one uses the Sweedler notation in weak multiplier Hopf algebras everywhere.

ii) If  the element $E$ exists and satisfies the first condition, we can show that the coproduct
 $\D$ has a unique extension to a homomorphism $\tilde{\D}: M(A)\lr M(A\o A)$.
  We denote the extension still by $\D$. In a similar way, we can extend
    $\D\o \i$ and $\i\o \D$ to the homomorphisms from $M(A\o A)$ to $M(A\o A\o A)$.
    So we can give a meaning to the formulas $(\i\o \D)E$ and $(\D\o \i)E$.
     We will call $E$ the {\it canonical idempotent}. It is uniquely determined and it satisfies
$$
E\D(a)=\D(a)=\D(a)E \eqno(2.2)
$$
for all $a\in A$.
\end{remark}

There is a unique antipode $S$ from $A$ to $M(A)$. Recall from Proposition 3.5 and 3.7 in \cite{VDW1} that
 the antipode $S$  is an anti-algebra and an anti-coalgebra  map. Moreover the antipode satisfies
 $\sum S(a_{(1)})a_{(2)}S(a_{(3)})=S(a)$ and
$$
\sum a_{(1)}S(a_{(2)})a_{(3)}=a \eqno(2.3)
$$
for all $a\in A$.  If $A$ is regular then the antipode is a bijective map from $A$ to itself
 (see Definition 4.5 in  \cite{VDW1}).

Let $(A,\D, \v, E, S)$ be a regular weak multiplier Hopf algebra. For any $a\in A$, we have the following four linear maps
 $\v _s, \v_t, \v'_s$ and $\v _t'$ from $A$ to $M(A)$ (see Proposition 2.1 and Remark 2.22 i) in \cite{VDW3}):
 $$
\left\{
\begin{array}{l}
\v_s(a)=(\i\o \v)((1\o a)E)=\sum S(a_{(1)})a_{(2)},\\
\v_t(a)=(\v\o \i)(E(a\o 1))=\sum a_{(1)}S(a_{(2)}),\\
\v'_s(a)=(\i\o \v)(E(1\o a))=\sum a_{(2)}S^{-1}(a_{(1)}),\\
\v'_t(a)=(\v\o \i)((a\o 1)E)=\sum S^{-1}(a_{(2)})a_{(1)}.
\end{array}
\right.
$$

 We will call $\v_s(A)$ the {\it source algebra} and $\v_t(A)$ the {\it target algebra} as in  \cite{VDW3}.
 They can be identified resp. with the left and the right leg of $E$. We have
 $\v_s(a)\v_s(b)=\v_s(a\v_s(b))$ and
$$
  \v_t(a)\v_t(b)=\v_t(\v_t(a)b),  \eqno(2.4)
$$
 where $a,b\in A$. In the regular case, they embed in $M(A)$ in such a way that their
 multiplier algebras $M(\v _s(A))$ and $M(\v _t(A))$ still embed in $M(A)$.
 These multiplier algebras are denoted by $A_s$ and $A_t$ resp.
 They are still commuting subalgebras of $M(A)$.

For a regular weak multiplier Hopf algebra, the multiplier algebras of
 the source and target algebras satisfy
 \begin{eqnarray*}
 && A_s=\{y\in M(A)\mid \D(ay)(1\o b)=\D (a)(1\o yb) \, \, \mbox{for all}\, \, a, b\in  A \},\\
 && A_t=\{x\in M(A)\mid (c\o 1)\D(xa)=(cx\o 1)\D (a)\, \, \mbox{for all}\, \, a, c\in  A \}.
 \end{eqnarray*}
For element $y\in  A_s$ and $x \in A_t$, we  have
 $\D (y) = E(1\o y)=(1\o y)E$ and
 $$
\D (x) = (x\o 1)E=E(x\o 1).  \eqno(2.5)
$$

For all $a\in A$, we also have from Proposition 2.7 in  \cite{VDW3} that
 $\v_s(ay)=\v_s(a)y, \v_t(ya)=\v_t(a)S(y), \v_t(xa)=x\v_t(a)$ and $\v_s(ax)=S(x)\v_s(a)$.

We also list some formulas here. For any regular  weak multiplier Hopf algebra $A$, we have
 (see \cite{VDW3}):
$$
E(a\o 1)=\sum \D (a_{(1)})(1\o S(a_{(2)}))=\sum a_{(1)}\o \v_t(a_{(2)})  \eqno(2.6)
$$
$$
 (1\o a)E=\sum (S(a_{(1)})\o 1)\D (a_{(2)})=\sum \v_s(a_{(1)})\o a_{(2)} \eqno(2.7)
$$
$$
E(1\o a)=\sum \D (a_{(2)})(S^{-1}(a_{(1)})\o 1)=\sum \v'_s(a_{(1)})\o a_{(2)} \eqno(2.8)
$$
for all $a\in A$. For any $ y\in A_s$, we have:
$$
E(y\o 1)=E(1\o S(y)).   \eqno(2.9)
$$

We now make an important remark about the covering of the previous formulas (see \cite{VDW3}).

\begin{remark}
i) First rewrite the (images of the) canonical maps $T_1$ and $T_2$,
 and of $T_3$ and $T_4$ in the regular case, using the Sweedler notation, as
$$
 \D (a)(1\o b) =\sum a_{(1)}\o a_{(2)}b   \quad \mbox{and} \quad (c\o 1)\D (a) =\sum ca_{(1)}\o a_{(2)} \eqno(2.10)
$$
and
$$
(1\o b)\D (a) =\sum a_{(1)}\o  ba_{(2)} \quad \mbox{and} \quad  \D (a)(c\o 1) =\sum a_{(1)}c\o a_{(2)}
$$
where $a, b, c \in A$. In all these four expressions, either $a_{(1)}$ is covered by $c$ and or
 $a_{(2)}$ by $b$. This is by the assumption put on the coproduct, requiring that the canonical maps have range
  in $A\o A$.

If  we first apply $S$ in the first or the second factor of the expressions in the formulas (2.10) and then multiply, we get the two elements
$$
 \sum S(a_{(1)})a_{(2)}b   \quad \mbox{and} \quad \sum ca_{(1)} S(a_{(2)})
$$
where $a, b, c \in A$. This is used to define the source and target maps above.

ii) Next consider the expressions
$$
 \sum a_{(1)}\o S(a_{(2)})b  \quad \mbox{and} \quad \sum cS(a_{(1)})\o a_{(2)} \eqno(2.11)
$$
and
$$
\sum a_{(1)}\o  bS(a_{(2)}) \quad \mbox{and} \quad  \sum S(a_{(1)})c\o a_{(2)}
$$
where $a, b, c \in A$.  In the first two formulas, we have a covering by the assumption that
 the generalized inverses $ R_1$ and $R_2$ of the canonical maps exist as maps on $A\o A$ with range
 in $A\o A$ (see \cite{VDW2}). In the second pair of formulas, we have a good covering only
 in the regular case.

If  we simply apply multiplication on the expressions
 in the formulas (2.11), we get the two elements
 $$
\sum a_{(1)} S(a_{(2)})b \quad \mbox{and} \quad  \sum cS(a_{(1)}) a_{(2)}
$$
where $a, b, c \in A$. This is also used to define the source and target maps above.

iii) Now, we combine the coverings obtained in the part i) and the part ii). Consider e.g. the two expressions:
\begin{eqnarray*}
&& \sum \D (a_{(1)})(1\o S(a_{(2)})b) \, \, \mbox {and}\, \,  \sum (cS(a_{(1)})\o 1) \D (a_{(2)})
 \end{eqnarray*}
where $a, b, c \in A$.  The first expression above is obtained by applying the canonical map $T_1$ to
 the first of the two expressions in (2.11). So this gives an element in
  $A\o A$ and we know that it is $E(a\o b)$ as we can see from the formula (2.6) above.
   Similarly, the second expression above is
 obtained by applying the canonical map $T_2$ to the second of the two expressions  in (2.11).
  We know that this is $(b\o a)E$ as we see from the formula (2.7) above. Remark that $ E(a\o b)$ and
 $(b\o a)E$ belong to $A\o A$ because by assumption $ E \in  M(A\o A)$, but that on the other hand,
 it is not obvious (as we see from the above arguments) that the expressions that we obtain for
 these elements belong to $A\o A$.

iv) Finally, as a consequence of the above statements, also the four expressions
\begin{eqnarray*}
&& \sum S(a_{(1)})a_{(2)}S(a_{(3)})b   \quad \mbox{and} \quad \sum ca_{(1)} S(a_{(2)})a_{(3)};\\
&& \sum a_{(1)} S(a_{(2)})a_{(3)}b \quad \mbox{and} \quad  \sum cS(a_{(1)}) a_{(2)}S(a_{(3)})
 \end{eqnarray*}
are well-defined in $ A$ for all $a, b, c \in A$. This
 justifies a statement made earlier about the properties of the antipode (see \cite{VDW3}).

And once again, in all these cases, the Sweedler notation is just used as a more transparent way to
denote expressions. We refer to the coverings just to indicate how the formulas with the Sweedler
notation can be rewritten without the use of it.
\end{remark}

\section{Module algebras over weak multiplier Hopf algebras}\label{three}
\subsection{Definition and module extension}
In this section, we fix a  weak multiplier Hopf algebra $(A, \D)$. We do not assume that it is regular.  We know that $A$ still admits local units  in the non-regular case, see Proposition 2.21 in \cite{VDW3}.

By a left $A$-module  we mean a vector space $R$ equipped with a bilinear map $A\o R\lr R, a\o r\mapsto ar$ satisfying $(ab)r=a(br)$ for all $r\in R$ and $a, b\in A$. It is called {\it unital} if $AR=R$. The module is called non-degenerate if $r\in R$ and $ar=0$ for all $a\in A$, implies $r=0$.   A unital module is automatically non-degenerate because we have local units. In this case, one can show the following.

\begin{proposition}
If $R$ is a unital left $A$-module, then it is  non-degenerate.
\end{proposition}

Let us give an important remark about the unital module which is from Section 3 in \cite{DVZ}.
We just copy it here because it will help the reader to understand the Sweedler notation or the covering.

\begin{remark}
 Let $R$ be a unital $A$-module. Since $A$ has local units, then there exists an  $e\in A$ such that $ex=x$ for any $x\in R$. Moreover, for all $x_1, x_2,..., x_n\in R$, we have an element $e$ such that $ex_i=x_i$ for all $i$. It means that elements in a unital $A$-module will cover elements $a_{(k)}\in A$.
\end{remark}

 Next we show that unital $A$-modules can be extended to modules over $M(A)$. This will help to explain the formulas in Definition \ref{3.2}.

\begin{theorem}\label{3.1}
Let $R$ be a unital left $A$-module. Then there is a unique extension to a left $M(A)$-module with $1x=x$ for all $x\in R$, here $1\in M(A)$.
\end{theorem}

For the proof we refer to Proposition 3.3 in \cite{DVZ}. Note that if $A$ does not have local units, we can also get the above result. We only need $R$ to be unital and non-degenerate. Let $r\in R, m\in M(A)$. Since $R$ is unital, we can write $r=\sum a_ix_i$ where $a_i\in A$ and $x_i\in R$ for all $i$. Then we can define the action of $M(A)$ by $mr=\sum (ma_i)x_i$. It is well-defined because the module is non-degenerate.

If $R$ and $S$ are unital left $A$-modules, then we can make $R\o S$ into a unital left $(A\o A)$-module by $(a\o b)(r\o s)=ar\o bs$. By the above theorem we can extend it to a module over the multiplier algebra $M(A\o A)$. Now we can consider the action of $A$ on $R\o S$ by  $a(r\o s)=\D(a)(r\o s).$  But the action is not unital any more since $T_1$ and $T_2$ are not surjective. Fortunately  we can show that the subspace $E(R\o R)$ is a non-degenerate and unital module under the action. This is because $\D(A)(A\o A)=E(A\o A)$.

\begin{definition}\label{3.2}
 Let $R$ be an algebra and  $R$ a unital left $A$-module via $a\o x\mapsto a\rhd x$
  for all $a\in A, x\in R$. Then
 $R$ is called a left $A$-module algebra if the following condition holds:
\begin{equation}
a\rhd (xy)=\sum(a_{(1)}\rhd x)(a_{(2)}\rhd y)
\end{equation}
for all $a\in A, x,y\in R$.
\end{definition}

Since $R$ is  a unital left $A$-module  we know that  the elements $x$ and $y$ can be used to cover
 $a_{(1)}$ and $a_{(2)}$. We can also explain this expression as
\begin{equation}
a\rhd m(x\o y)=m(\D(a)\rhd (x\o y))
\end{equation}
 where $m$ denotes multiplication in $R$.

\begin{proposition}\label{3.3}
Let $R$ be a left $A$-module algebra. The following properties hold for all $r,r'\in R$ and $y\in \v_s(A)$,
\begin{equation}
m(E\rhd (r\o r'))=rr'
\end{equation}
and
\begin{equation}
(y\rhd r)r'=r(S(y)\rhd r').
\end{equation}
\end{proposition}

\begin{proof}
 (i) Note that  $E\in M(A\o A)$, so the formula (3.3)  is meaningful. For any $a\in A$,
$$
a\rhd m(E\rhd (r\o r'))\overset{(3.1)}=m(\D(a)E\rhd (r\o r'))\overset{(2.2)}=m(\D(a)\rhd (r\o r'))\overset{(3.2)}=a\rhd (rr').
$$
By the non-degeneracy of the module, we get $m(E\rhd (r\o r'))=rr'$.

 (ii) Remark that by Theorem \ref{3.1} we have a left $M(A)$-module, so the formula (3.4) is meaningful.
 Take any element $y$ in $\v_s(A)$,  for any $a\in A$ we have
 \begin{eqnarray*}
a\rhd ((y\rhd r)r') &\overset{(3.1)}{=}& \sum((a_{(1)}y)\rhd r)(a_{(2)}\rhd r')\\
                    &\overset{(2.9)}=& \sum(a_{(1)}\rhd r)((a_{(2)}S(y))\rhd r')\\
                    &\overset{(3.1)}{=}& a\rhd (r(S(y)\rhd r')).
\end{eqnarray*}
 Because the module is non-degenerate we find $(y\rhd r)r'=r(S(y)\rhd r')$.
\end{proof}

Assume that $R$ has a unit $1_R$ and $A$ is regular. Then we find
\begin{eqnarray*}
a\rhd 1_R=(a\rhd 1_R)1_R   &\overset{(3.3)}{=}& m(E(a\o 1)\rhd (1_R\o 1_R))\\
                     &\overset{(2.6)}{=}& \sum m((a_{(1)}\rhd 1_R)\o a_{(2)}S(a_{(3)})\rhd 1_R)\\
                     &\overset{(3.1)}{=}& \sum a_{(1)}\rhd (1_R(S(a_{(2)})\rhd 1_R))\\
                     &=& \sum a_{(1)}\rhd (S(a_{(2)})\rhd 1_R)\\
                     &=& \v_t(a)\rhd 1_R.
\end{eqnarray*}
In the fourth and fifth equalities $a_{(2)}$ is covered by $1_R$.

In the case of a multiplier Hopf algebra, this result means $a\rhd 1_R = \v(a)1_R$, which is a true and known result.
The above one generalizes this to weak multiplier Hopf algebras.

Let $A$ be a multiplier Hopf algebra and assume that $R$ is a left $A$-module algebra. In \cite{DVZ} the authors have extended the action of $A$
 on $R$ to the multiplier algebra $M(R)$. Now we will generalize the theory to regular weak multiplier Hopf algebras.

\begin{proposition}
Let $A$ be  a regular weak multiplier Hopf algebra.
For any $a\in A$ and $r,r'\in R$, we have
$$
(ar)r'=\sum a_{(1)}(rS(a_{(2)})r'),
$$
and
$$
r(ar')=\sum  a_{(2)}((S^{-1}(a_{(1)})r)r').
$$
\end{proposition}

\begin{proof}
First remark that $a_{(2)}$ is covered by $r'$ in the first formula and that $a_{(1)}$ is covered by $r$. For the first formula we have
\begin{eqnarray*}
\sum a_{(1)}(rS(a_{(2)})r')&\overset{(3.1)}=& \sum (a_{(1)}r)(\v_t(a_{(2)})r')\\
                           &\overset{(3.4)}=& \sum (S^{-1}(\v_t(a_{(2)}))a_{(1)}r)r'\\
                           &\overset{(2.3)}=& \sum (ar)r'.
\end{eqnarray*}

For the second equality we have
\begin{eqnarray*}
\sum  a_{(2)}((S^{-1}(a_{(1)})r)r')&\overset{(3.1)}=& \sum (\v'_s(a_{(1)})r)(a_{(2)}r')\\
                                    &\overset{(2.8)}=& (E_{(1)}r)(E_{(2)}ar')\\
                                    &\overset{(3.3)}=& r(ar'),
\end{eqnarray*}
where we use the Sweedler type notation: $E = E_{(1)}\o  E_{(2)}$.
\end{proof}

Now inspired by the proposition above we can extend the action of $A$ from $R$ to $M(R)$.

\begin{proposition}
 Let $A$ be a  regular weak multiplier Hopf algebras and assume that $R$ is a left $A$-module algebra.
 Then we can extend this action of $A$ from $R$ to $M(R)$. Also $a1=\v_t(a)$ for any $a\in A$.
\end{proposition}

\begin{proof}  Let $a\in A, m\in M(R), r\in R$, we will  define $am\in M(R)$ by
$$(am)r=\sum a_{(1)}(m(S(a_{(2)})r))$$
$$r(am)=\sum a_{(2)}((S^{-1}(a_{(1)})r)m).$$

As in the above proposition, we have well coverings for the two expression. Next we first show that the action is well-defined, it means that we need to prove that $am$ is actually a multiplier in $M(R)$. For this, let $r, s\in R$, we have
$$(r(am))s=\sum (a_{(2)}((S^{-1}(a_{(1)})r)m))s.$$
On the other hand we have
\begin{eqnarray*}
r((am)s) &=& \sum r(a_{(1)}(m(S(a_{(2)})s))) \\
          &=& \sum a_{(2)}((S^{-1}(a_{(1)})r)(m(S(a_{(3)})s))) \\
          &=& \sum (a_{(2)}((S^{-1}(a_{(1)})r)m))(\v_t(a_{(3)})s)) \\
          &\overset{(3.4)}=& \sum (a_{(2)}((S^{-1}(a_{(1)})r)m))s.
\end{eqnarray*}
 Next we will show that  $M(R)$ is a  left $A$-module.
Let $a,b\in A, m\in M(R)$. Then for any $r\in R$, we have
\begin{eqnarray*}
((ab)m)r = \sum (ab)_{(1)}(m(S((ab)_{(2)})r))   &=& \sum a_{(1)}b_{(1)}(m(S(b_{(2)})S(a_{(2)})r))\\
                                                &=& \sum a_{(1)}((bm)(S(a_{(2)})r))\\
                                                &=& (a(bm))r.
\end{eqnarray*}

If $m=1$, we have
$$(a1)r=\sum a_{(1)}(S(a_{(2)})r)=\sum (a_{(1)}S(a_{(2)}))r=\v_t(a)r,$$
so we have $a1=\v_t(a)$ for any $a\in A$.
\end{proof}

 So far we know that we can make $M(R)$ into a left $A$-module, but we can not make sure it is still a module algebra since $M(R)$ is not unital any more. However we do  know
that the action is  non-degenerate.  If $m\in M(R)$  and $am=0$ for all $a$, by the definition we have $\sum a_{(1)}(m(S(a_{(2)})x))=0$ for all $x\in R$. We can write $x$ as $\sum S(b_i)x_i$($b_i\in A, x_i\in R $) and use the fullness of $\D$, then we can get the non-degeneracy.

\subsection{Examples}
Now let us treat some examples and special cases.
\subsubsection{The trivial action}
 Let $A$ be a weak multiplier Hopf algebra. Then $\v_t(A)$ is an $A$-module algebra with the module action
$$a\rhd  \v_t(b)=\v_t(a\v_t(b))=\v_t(ab)$$
for all $a,b\in A$.

It is easy to show that $\v_t(A)$ is an $A$-module.  Since $A$ is idempotent, we know that $\v_t(A)$ is a unital $A$-module.  Take $a,x,y\in A$, we have
$$a\rhd (\v_t(x)\v_t(y))=a(\v_t(\v_t(x)y))=\v_t(a\v_t(x)y).$$
On the other hand,
\begin{eqnarray*}
\v_t(a_{(1)}x)\v_t(a_{(2)}y )& \overset{(2.4)}= & \v_t(a_{(1)}x_{(1)}S(x_{(2)})S(a_{2})a_{(3)}y) \\
                            & = & \v_t(a_{(1)}\v_t(x)\v_s(a_{(2)})y)\\
                            & = & \v_t(a_{(1)}\v_s(a_{(2)})\v_t(x)y)\\
                            & \overset{(2.3)}= & \v_t(a\v_t(x)y).
\end{eqnarray*}
In the first equality $a_{(3)}$ is covered by $y$. The third equality follows by the commutative of the two base algebras
 $\v_s(A)$ and $\v_t(A)$.  So $\v_t(A)$ is an $A$-module algebra.

Moreover  we can  check the formula (3.3) in Proposition \ref{3.3}. For any $a,b,c,d\in A$ we have
\begin{eqnarray*}
m E\rhd (\v_t(ab)\o\v_t(cd))  & = & m((E(a\o c))\rhd (\v_t(b)\o\v_t(d))) \\
                           & \overset{(2.6)}= &\v_t(a_{(1)}b)\v_t(a_{(2)})\v_t(cd)\\
                           & \overset{(2.4)}= & \v_t(\v_t(a_{(1)}b)a_{(2)})\v_t(cd)\\
                           & = & \v_t(a_{(1)}\v_s(a_{(2)})\v_t(b))\v_t(cd)\\
                           & \overset{(2.3)}= & \v_t(ab)\v_t(cd).
\end{eqnarray*}
In the second equality $a_{(1)}$ is covered by $b$. The fourth one follows because $\v_t(A)$ and $\v_s(A)$ are commuting algebras.

For the second formula (3.4), we have
\begin{eqnarray*}
\v_t(a)(S(\v_s(c))\rhd \v_t(b)) & = &  \v_t(a)\v_t(\v_t(S(c))\v_t(b)) \\
                           & \overset{(2.4)}= & \v_t(a)\v_t(S(c))\v_t(b) \\
                           & = &   \v_t(\v_t(\v_s(c)a)b) \\
                           & \overset{(2.4)}= &  (\v_s(c)\rhd \v_t(a))\v_t(b).
\end{eqnarray*}
The first equality follows by $S\circ \v_s=\v_t\circ S$.  The third one follows by $\v_t(a)S(y)=\v_t(ya)$.

Now let us consider the module extension. From the above proposition and $M(\v_t(A))=A_t$, we know that we can extend the action of $A$ from $\v_t(A)$ to $A_t$.  For any $a, b, c\in A_t, x\in A_t$,  we have
\begin{eqnarray*}
x\v_t(a)= \v_t(xa)    &=& \sum x_{(1)}a_{(1)}S(a_{(2)})S(x_{(2)})\\
                      &\overset{(2.5)}=& \sum E_{(1)}x\v_t(a)S(E_{(2)})\\
                      &=& \sum E_{(1)}(x(S(E_{(2)})\v_t(a))),
\end{eqnarray*}
where we use the Sweedler type notation: $E = E_{(1)}\o  E_{(2)}$.
In the third equality  we have a covering by multiplying any element in $A$ from the left or right. The last equality follows by the commutative of the two base algebras and the fact that the restriction of  $S$ is an anti-isomorphism from $\v_t(A)$ to $\v_s(A)$(see Proposition 2.16 in \cite{VDW3}).

From this formula we can get the non-degeneracy of the extended module. In fact, if $A$ is a weak Hopf algebra, the formula means $1x=x$.

\subsubsection{The groupoid case}
Let $G$ be any groupoid and consider the groupoid algebra $\mathbb{C}G$. It is the space of complex functions with finite support on $G$ with the convolution product. Denote the canonical embedding of $G$ in $\mathbb{C}G$ by $p\mapsto \lambda_p$. When $pq$ is defined we have $\l_p\l_q=\l_{pq}$, otherwise the product is 0. The coproduct on $A$ is given by $\D(\l_p)=\l_p\o\l_p$ and the counit is given by $\v(\l_p)=1$ for all $p$. The canonical idempotent $E$ in $M(A\o A)$ is defined as $\sum\l_e\o \l_e$ where the sum is taken over all units.

Let $\a$ be an  action of $G$ on an set $X$. It means that for every $p\in G$ there is a subset $X_p$ of $X$ and a map $\a_p: X_p\lr X$ such that:
\begin{itemlist}
 \item If $pq$ exists then $X_q\subseteq X_{pq}$, $\a_q(X_q)\subset X_p$ and $\a_p\a_q(x)=\a_{pq}(x)$ for any $x\in X_q$,
 \item If $e$ is a unit in $G$ then $\a_e(x)=x$ for all $x\in X_e$.
\end{itemlist}

We also assume that the action is true.  The action is called true if $s\in X_p$ and $\a_p(s)\in X_q$ imply $pq$ is defined. For more information about the notion of an action  of groupoid on a set we refer to \cite{R}.

By the definition we have $\bigcup\limits_{p\in G}X_p=X$, moreover $\bigcup\limits_e X_e=X$ where the union is taken over the set of units. Let $R$ be the algebra $K(X)$  of complex functions with finite support on $X$ and pointwise product. For each $p\in G$,  $\g_p$ is a map from $R$ to $R$ which is defined as
$$
\g_p(f)(x)=\left\{
               \begin{array}{ll}
                 f(\a_{p^{-1}}(x)), & \hbox{if $x\in X_{p^{-1}}$,} \\
                 0, & \hbox{otherwise.}
               \end{array}
             \right.
$$

For any $f\in R,x\in X$. If $pq$ is defined, we have
\begin{eqnarray*}
\g_p\g_q(f)(x) &=& \left\{
                     \begin{array}{ll}
                     \g_q(f)(\a_{p^{-1}}(x)), & \hbox{if $x\in X_{p^{-1}}$,} \\
                     0, & \hbox{otherwise.}
                     \end{array}
                   \right.
\end{eqnarray*}

When $x\in X_{p^{-1}}$, we have $\a_{p^{-1}}(x)\in X_{q^{-1}}$. So
$$\g_p\g_q(f)(x)=\g_q(f)(\a_{p^{-1}}(x))=f(\a_{q^{-1}}(\a_{p^{-1}}(x)))=f(\a_{q^{-1}p^{-1}}(x))=\g_{pq}(f)(x).$$

When $x\notin X_{p^{-1}}$,  we  have
$$\g_p\g_q(f)(x)=0=\g_{pq}(f)(x).$$

If $pq$ is not defined and $x\in X_{p^{-1}}$. Since the action is true we get $\a_{p^{-1}}(x)\notin X_{q^{-1}}$. So we can also get $\g_p\g_q(f)=\g_{pq}(f)$.

So we finally get
$$\g_p\g_q=\left\{
             \begin{array}{ll}
               \g_{pq}, & \hbox{if $pq$ is defined,} \\
               0, & \hbox{otherwise .}
             \end{array}
           \right.
$$

Now we can  associate an action of $A$ on $R$ by $\l_p\triangleright f=\g_p(f)$ for $p\in G, f\in R$. Next we show that $R$ is an $A$-module algebra.
 It is easy to see that $R$ is an $A$-module. Let $e$ be a unit in $G$, we have
\begin{eqnarray*}
\g_e(f)(x) &=& \left\{
               \begin{array}{ll}
                 f(\a_{e}(x)), & \hbox{if $x\in X_{e}$,} \\
                 0, & \hbox{otherwise}
               \end{array}
             \right.\\
            &=& \left\{
               \begin{array}{ll}
                 f(x), & \hbox{$x\in X_e$,} \\
                 0, & \hbox{otherwise.}
               \end{array}
             \right.
\end{eqnarray*}
If we take $f$ with support in $X_e$, then $\g_ef=f$. So the action is unital as $\bigcup\limits_e X_e=X$.
Finally, for any $f,g\in R, x\in X$,
$$\l_p(fg)(x)=\g_p(fg)(x)=\left\{
                           \begin{array}{ll}
                             (fg)(\a_{p^{-1}}(x)), & \hbox{if $x\in X_{p^{-1}}$,} \\
                             0, & \hbox{otherwise.}
                           \end{array}
                         \right.
$$
On the other hand,
$$((\l_pf)(\l_pg))(x)=(\l_pf)(x)(\l_qg)(x)=\left\{
                           \begin{array}{ll}
                             f(\a_{p^{-1}}(x))g(\a_{p^{-1}}(x)), & \hbox{if $x\in X_{p^{-1}}$,} \\
                             0, & \hbox{otherwise.}
                           \end{array}
                         \right.$$
So $R$ is a $\mathbb{C}G$-module algebra.

\subsubsection{The adjoint action}
Now we consider  a regular weak multiplier Hopf algebra $A$. Take $a\in A$ and define a map $\a_a: A\lr A$  by $\a_a(x)=\sum a_{(1)}xS(a_{(2)})$ where $x\in A$. Observe that $a_{(1)}$ is covered by $x$.

In what follows, we will still use the Sweedler type notation: $E = E_{(1)}\o  E_{(2)}$. Let
$$
A_0=\text{span}\{E_{(1)}pS(E_{(2)}q)\mid p,q\in A\}.
$$
 Since  we have $\D(A)(A\o 1)=E(A\o A)$ and $AS(A)=A$, so
$$
A_0=\text{span}\{E_{(1)}aS(E_{(2)})\mid a\in A\}=\text{span}\{\a _a(x)\mid a, x\in A\}.
$$
Note that  $E(a\o 1)\in A\o \v_t(A)$ (see Remark 2.22 in \cite{VDW3}). Recall from Proposition 2.16 in \cite{VDW3} that we have
$$
E_{(1)}aS(E_{(2)})\in AS(\v_t(A))\subseteq A\v_s(A)\subseteq A.
$$

Define a linear map $\p: A\lr A_0, a\mapsto E_{(1)}aS(E_{(2)})$. By definition $\p$ is surjective.

\begin{proposition}
$A_0$ is a subalgebra of $A$.
\end{proposition}

\begin{proof}
Take  $a, b\in A$. We write  $\p(a)= E_{(1)}aS(E_{(2)})$ and  $\p(b)=E'_{(1)}bS(E'_{(2)})$ where we use two copies $E=E_{(1)}\o E_{(2)}$ and $E=E'_{(1)}\o E'_{(2)}$ of $E$. Then we have
\begin{eqnarray*}
\p(a)\p(b)=E_{(1)}a S(E_{(2)})E'_{(1)}bS(E'_{(2)}) &=& E_{(1)}a S(S^{-1}(E'_{(1)})E_{(2)})b S(E'_{(2)})\\
                                                    &=& E'_{(1)}E_{(1)}a S(E_{(2)})b S(E'_{(2)})\\
                                                    &=& E'_{(1)}\p(a)b S(E'_{(2)}) \\
                                                    &=& \p(\p(a)b).
\end{eqnarray*}
In the third equality we use
$$(y\o 1)E=(1\o S^{-1}(y))E$$
where $y$ is in the source algebra $\v _s(A)$. Note that we  have everything well-covered here. In the first equality $E_{(1)}$ is covered by $a$ and $E'_{(1)}$ is covered by $b$. In the third equality   the element $a$ covers $E_{(1)}$  and $b$ covers $E'_{(2)}$.
\end{proof}

\begin{proposition}\label{3.4}
$A_0A=AA_0=A.$
\end{proposition}

\begin{proof}
 Note that $A_0=\text{span}\{\sum a_{(1)}xS(a_{(2)})\mid a,x\in A \}$, so  one gets
 $$
 AA_0=\text{span}\{\sum ba_{(1)}xS(a_{(2)})\mid a,b,x\in A \}.
 $$
  Let $x$ be the local unit of $A$. Then we have
$$
\text{span}\{\sum ba_{(1)}S(a_{(2)})\mid a,b\in A \} \subseteq \text{span}\{\sum ba_{(1)}xS(a_{(2)})\mid a,b,x\in A \}.
$$
The left-hand side is $A\v_t(A)$ and because $A\v_t(A)=A$ we find that $AA_0=A$. Similarly for the other equality.
\end{proof}

\begin{proposition}\label{3.5}
The product in $A_0$ is non-degenerate.
\end{proposition}

\begin{proof}
Let $a\in A_0$ and assume that $ab=0$ for all $b\in A_0$, then $abc=0$ for all $c\in A$. It means that $ax=0$ for all $x\in A$, so $a=0$.
\end{proof}

Moreover, we have the following propositions.

\begin{proposition}
For any $a\in A, y\in \v_s(A)$, we have $y\p(a)=\p(a)y$.
\end{proposition}

\begin{proof}
$y\p(a)=yE_{(1)}a S(E_{(2)})=E_{(1)}aS(S^{-1}(y)E_{(2)})=\p(a)y.$
\end{proof}

\begin{proposition}
Let $a\in A$ and assume that   $a$ commutes with $\v_s(A)$. Then  $\p(a)=a$ and $\p(ab)=a\p(b)$ and $\p(ba)=\p(b)a$ for all $b\in A$.
\end{proposition}

It is easy to prove it. So we have that $\p$ is a conditional expectation of $A$ onto $A_0$. Now we can give a characterization of $A_0$ and  $M(A_0)$.

\begin{proposition} We have the following identities:
$$A_0=\{a\in A \mid ay=ya, \forall y\in \v_s(A)\},$$
and
$$M(A_0)=\{m\in M(A) \mid my=ym, \forall y\in \v_s(A)\}.$$
\end{proposition}

\begin{proof}
The first one is a consequence of Proposition \ref{3.4} and Proposition \ref{3.5}. Now let us
 consider the multiplier algebra of $A_0$. We first want to show the inclusion "$\subseteq$". So take
  $m\in M(A)$ and assume that $y\in \v_s(A)$. For any $a\in A_0, b\in A$, we have
$$(my)a=m(ya)=m(ay)=(ma)y=y(ma)=(ym)a.$$
 If we multiply any element in $ A$ from the left and use $A_0A=A$, then we get $(my)b=(ym)b$. So $my=ym$.

Next we will show the inclusion "$\supseteq$". Take $a\in A_0$ and $y\in \v_s(A)$, then
$$(ma)y=m(ay)=m(ya)=y(ma).$$
It means that $ ma\in A_0$. Similarly, we can get $am\in A_0$. Hence $m\in M(A_0)$.
\end{proof}

\begin{proposition}[adjoint action]
$A_0$ is an $A$-module algebra with the action $\a: A\o A_0\lr A_0, \a_a(x)=\sum a_{(1)}xS(a_{(2)})$, for any $a\in A, x\in A_0$.
\end{proposition}

\begin{proof}
Obviously $A_0$ is a unital left $A$-module. For all $x,y\in A_0$,
$$\sum (a_{(1)}xS(a_{(2)}))(a_{(3)}yS(a_{(4)}))=\sum a_{(1)}x\v_s(a_{(2)})yS(a_{(3)})=\sum a_{(1)}xyS(a_{(2)}).$$
 In the expression $a_{(1)}xS(a_{(2)})$, we observe that $a_{(1)}$ is covered by $x$.
\end{proof}

\subsubsection{The example associated with a separability idempotent}

Now let us consider the example associated with a separability idempotent which is studied in \cite{VD3} and \cite{VDW3}.

Let $B$ and $C$ be non-degenerate algebras and assume that $E$ is a regular separability idempotent in $M(B\o C)$. We call $E$ regular if it is a separability idempotent also when considered in $M(B^{op}\o C^{op})$. Consider the algebra $P=C\o B$ with the coproduct $\D(c\o b)=c\o E\o b$, where $c\in C, b\in B$. Then $(P,\D)$ is a regular weak multiplier Hopf algebra. The canonical idempotent $E_P$ is $1\o E\o 1$. Let $R$ be a left $P$-module algebra with action denoted by $\rhd$. Then $R$ can be regarded as a $C$-module and a $B$-module through the following actions
$$c\rhd_Cr=(c\o 1)\rhd r \quad , \quad b\rhd_B r=(1\o b)\rhd r $$
for any $c\in C, b\in B, r\in R$. For any $s\in R$ we also have
$$(c\o b)\rhd r=c\rhd_C(b\rhd_B r)=b\rhd_B(c\rhd_C r)$$
and
\begin{eqnarray*}
(c\o b)\rhd rs &=& ((c\o E_{(1)})\rhd r)((E_{(2)}\o c)\rhd s)\\
               &=&  (E_{(1)}\rhd_B(c\rhd_C r))(E_{(2)}\rhd_C(b\rhd_B s))\\
               &=&  (c\rhd_C r)(b\rhd_B s).
\end{eqnarray*}
If we consider the extended $M(P)$-module and  the element $c\o 1$ in $M(P)$, then we get $c\rhd_C(rs)=(c\rhd_C r)s$. Similarly if we consider $1\o c$, then $b\rhd_C(rs)=r(b\rhd_B s)$. So there  exists a  left multiplier $\g_C(c)$ in $L(R)$ such that $c\rhd_C r=\g_C(c)r$ and a right multiplier $\g_B(b)$ in $R(R)$ such that $b\rhd_B r=r\g_B(b)$.

Note that $\v_s(P)=1\o B$ and $\v_t(P)=c\o 1$, so $(b\rhd_B r)s=r(S_B(b)\rhd_C s)$ for any $b\in B, c\in C, r,s\in R$. We can rewrite it as $(r\g_B(b))s=r(\g_C(S_B(b))s)$. Then $\g_B(b)$ is equal to $\g_C(S_B(b))$ as multipliers.

Now, we are ready to give the following proposition.

\begin{proposition}
As above, let $R$ be a left $P$-module algebra. For any $c\in C, b\in B, r\in R$, there exists a non-degenerate homomorphism $\g: C\lr M(R)$ such that
\begin{romanlist}[(iii)]
\item $c\rhd_C r=\g(c)r$,
\item $b\rhd_B r=r\g(S_B(b))$,
\item $(c\o b)\rhd r= \g(c)r\g(S_B(b)).$
\end{romanlist}
\end{proposition}

\subsubsection{The dual action}

Finally let us consider the examples that come from a dual pair of regular weak multiplier Hopf algebras.
 We will give the definition of a weak multiplier Hopf algebra pairing. Our idea is coming from \cite{DV}. In fact the definition is very similar to the case of multiplier Hopf algebras (see Definition 2.1 and Definition 2.8 in \cite{DV}).

 Let $(A, \D )$ be a weak multiplier Hopf algebra. Recall from \cite{KV} or \cite{VDW4} that
  a linear functional $\vp : A \lr  \mathbb{C}$ is called {\it left invariant}
 if $(\i \o \vp )\D(a) \in A_t$ for all $a \in A$.  A non-zero left invariant functional
  is called a {\it left integral}  on $A$;
  Similarly, a linear functional $\psi$  on $A$
 is called {\it right invariant} if $(\psi \o \i )\D (1)\in A_s$ for all $a \in A$.
 A non-zero right invariant functional is called a {\it right integral} on $A$.

Recall that $A_t$ and $A_s$ are defined as subspaces of $M(A)$ and so the above definition makes
sense.

\begin{definition}
Let $A$ and $B$ be two regular weak multiplier Hopf algebras with enough integrals. Define two linear functions $_af=\<a,\cdot\>\in B'$ and $f_b:= \<\cdot, b\>\in A'$ where $a\in A, b\in B$. A {\it pre-pairing} between $A$ and $B$ is a bilinear form $\langle,\rangle$ from $A\times B$ to $\mathbb{C}$ satisfying
 the following: for all $a,a'\in A, b,b'\in B$
\begin{eqnarray*}
&& (_af\o id)\D(b)\in B \quad \quad (id\o_af)\D(b)\in B,\\
&&(f_b\o id)\D(a)\in A \quad \quad (id\o f_b)\D(a)\in A,\\
&& {}_af(id\o _{a'}f)\D(b)=_{a'}f(_af\o id)\D(b)=_{aa'}f(b),\\
&& f_b(id\o f_b')\D(b)=f_{b'}(f_b\o id)\D(a)=f_{bb'}(a).
\end{eqnarray*}
The pre-pairing is called non-degenerate if $A$ and $B$ are dual with respect to the bilinear form.
\end{definition}

Since $ (_af\o id)\D(b)\in B$ we will denote it as $\sum \<a, b_{(1)}\>b_{(2)}$. Similarly for other cases. Remark that the Sweedler notation here is just a notation,  thus we can denote formulas in a more transparent way.

For any pre-pairing we have the following four maps
\begin{eqnarray*}
&&\varphi^l_{A,B}: A\o B\lr B: a\o b\mapsto \sum\<a, b_{(2)}\>b_{(1)}:=a\rhd b,\\
&& \varphi^r_{A,B}: B\o A\lr B: b\o a\mapsto \sum\<a, b_{(1)}\>b_{(2)}:=b\lhd a,\\
&& \varphi^l_{B,A}: B\o A\lr A: b\o a\mapsto \sum a_{(1)}\<a_{(2)},b\>:=b\rhd a,\\
&&\varphi^r_{B,A}: A\o B\lr A: a\o b\mapsto \sum a_{(2)}\<a_{(1)}, b\>:=a\lhd b.
\end{eqnarray*}
\begin{definition}
The  weak multiplier Hopf algebra pre-pairing $(A,B,\<,\>)$ is called a {\it pairing} if
 the four maps defined above are surjective and $\v_A(a)=\<a,1\>, \v_B(b)=\<1,b\>$, for any $a\in A, b\in B$.
\end{definition}

\begin{remark}
In multiplier Hopf algebras theory, if one of the four maps is surjective then so do the others. We have similar results for weak multiplier Hopf algebras. But the proof is not the same and it involves a long paragraph to explain. We will discuss it in a separate paper.
 So in the definition above we require these four maps to be surjective.

We also need to give a meaning to the formula $\<a,1\>$. We will explain it after the following proposition. And note that the formula is not involved in the following proposition.
\end{remark}

\begin{proposition}
These four maps $\varphi^l_{A,B},\varphi^r_{A,B},\varphi^l_{B,A},\varphi^r_{B,A}$ are actions, i.e. $(B,\varphi^l_{A,B})$ is a left $A$-module algebra and $(B,\varphi^r_{A,B})$ is a right $A$-module algebra. Analogously $(A,\varphi^l_{B,A})$ is a left $B$-module algebra and $(A,\varphi^r_{A,B})$ is a right $B$-module algebra.
\end{proposition}

\begin{proof}
Let us check the map $\varphi^l_{A,B}$. The action will be denoted by $\rhd$. It is easy to show that $B$ is a left $A$-module.
 Since the map is surjective we know that $B$ is  unital. So next we have to show that
$$a\rhd (bb')=\sum (a_{(1)}\rhd b)(a_{(2)}\rhd b')$$
for all $a\in A, b,b'\in B$. Indeed, for any $x\in A$,
\begin{eqnarray*}
\<x,\sum (a_{(1)}\rhd b)(a_{(2)}\rhd b')\> &=& \sum \<x_{(1)}, a_{(1)}\rhd b\>\<x_{(2)}, a_{(2)}\rhd b'\> \\
                                           &=& \sum \<x_{(1)}a_{(1)}, b\>\<x_{(2)}a_{(2)}, b'\>  \\
                                           &=& \<xa, bb'\>\\
                                           &=& \<x,a\rhd (bb')\>.
\end{eqnarray*}
Since the module is unital so $a_{(1)}$ is covered by $b$, then $x_{(1)}$ also be covered.  The other cases are similar.
\end{proof}

Since we have these four unital modules, the pairing on $A\times B$ can be uniquely extended to $A\times M(B)$ in such a way that
$$\<a,bm\>=\<a\lhd b, m\> \quad \text{and} \quad \<a,mb\>=\<b\rhd a, m\>$$
where $a\in A, b\in B, m\in M(B)$. So we can give a meaning to the pairing of the form $\<a,1\>$ in $A\times M(B)$.  Remark that we  have to show that this is well-defined. Similarly, the paring on $A\times B$ can be extended to $M(A)\times B$.

So for any $\<A, B\>$ we have the following useful formula
$$\<a\o a',E_B\>=\<a\o a',\D(1)\>=\v_A(aa')$$
where $a,a'\in A, 1\in M(B)$.

Next we will give the definition of the dual space $\hat{A}$ coming from Definition 2.8 in \cite{VDW4}.

\begin{definition}
Let $A$ be a regular weak multiplier Hopf algebra with a faithful set of integrals. Then we define $\hat{A}$ as the space of linear functionals on $A$ spanned by the elements of the form $\varphi(a \cdot )$ where $\varphi$ is a left integral of $A$ and $a\in A$.
\end{definition}

We say that $A$ has a faithful set of integrals if given an element $x\in A$ we must have $x=0$ if $\vp(xa)=0$ for any left integral $\vp$ and element $a\in A$. Similarly also if $\vp(ax)=0$ for any left integral $\vp$ and element $a\in A$, then $x=0$.
 For more information about the integrals we also refer to \cite{KV}.

If $A$ is a regular weak multiplier Hopf algebra with a faithful set of integrals we call it an {\it algebraic quantum groupoid} (see Definition 2.10 in \cite{VDW4}).  And the dual $\hat{A}$ of an algebraic quantum groupoid $A$ is a regular weak multiplier Hopf algebra (see Theorem 3.15 in \cite{VDW4}).
For the convenience of studying  the natural pairing $\<A, \hat{A}\>$
 in  Section 5 and Section 6, we here list  Theorem 3.15 in \cite{VDW4} as follows.

 \begin{theorem}
Let $A$ be an algebraic quantum groupoid. Then the dual pair $(\widehat{ A}, \widehat{\D})$ is a regular weak multiplier Hopf algebra
 with the following structure:
 \begin{eqnarray*}
&\bullet &  \mbox {product:}\, \, (\om \om ')(x)=(\om \o \om ')\D (x);\\
&\bullet & \mbox {counit}\, \,  \widehat{\v}: \widehat{A}\lr \mathbb{C} \, \, \mbox {defined by }\, \, \widehat{\v }(\om) = \om(1);\\
&\bullet & \mbox {coproduct:}\,\,  \widehat{\D }: \widehat{A}\lr M(\widehat{A}\o \widehat{A})\, \, \mbox {is determined by the following four identities:}\\
 && \quad \widehat{T_1}(\om '\o \om '')=\widehat{\D} (\om ')(1\o \om ''),  \quad \< x\o y, \widehat{T_1}(\om \o \om ')\>=\<T_2(x\o y), \om \o \om'\>,\\
 && \quad  \widehat{T_2}(\om \o \om ')=(\om \o 1)\widehat{\D} (\om '), \quad \< x\o y, \widehat{T_2}(\om \o \om ')\>=\<T_1(x\o y), \om \o \om'\>;\\
&\bullet & \mbox {antipode:}\,\, \widehat{S}:  \widehat{A} \lr \widehat{A}\, \, \mbox {defined by }\, \, \<x, \widehat{S} (\om)\>=\<S(x), \om\>;\\
&\bullet & \mbox {idempotent:}\,\, \widehat{E}\in M(\widehat{A}\o \widehat{A})\, \, \mbox {defined by }\, \, \< x\o y, \widehat{E}\>=\v (xy).
\end{eqnarray*}
for all  $x, y \in A$ and $\om , \om ', \om ''\in \widehat{A}$.
\end{theorem}

\begin{remark}
Let $A$ has a left integral $\vp$. Then in the pairing $\<A, \hat{A}\>$,  we note that $\<a, f\>$ means that $f(a)$ where $a\in A, f\in \hat{A}$. In particular,
 if $f=\varphi (x \cdot )$ with $x\in A$, then $\<a, f\>=\varphi (xa)$.
\end{remark}

\section{Smash products}\label{four}

Let $A$ be a regular weak multiplier Hopf algebra and assume $R$ is  a left $A$-module algebra. In
this section we will use $\rhd$ denote the action of $A$ on $R$. Define the product on
$R\o A$ by
\begin{equation}
(r\o a)(s\o b)=\sum r(a_{(1)}\rhd s)\o a_{(2)}b
\end{equation}
for any $r,s\in R, a,b\in A$. Here $a_{(1)}$ is covered by $s$. Then we have the following proposition.

\begin{proposition}
 The product given by (4.1) is associative.
\end{proposition}

\begin{proof}
For any $r,r',r''\in R$ and $a,a',a''\in A$,
\begin{eqnarray*}
((r'\o a')(r''\o a''))(r\o a)  &\stackrel {(4.1)}=& \sum (r'(a'_{(1)}\rhd r'')\o a'_{(2)}a'')(r\o a)\\
                                &\stackrel {(3.1)}=& \sum r'(a'_{(1)}\rhd(r''(a''_{(1)}\rhd r)))\o a'_{(2)}a''_{(2)}a\\
                                &\stackrel {(4.1)}=& (r'\o a')(\sum r''(a''_{(1)}\rhd r)\o a''_{(2)}a)\\
                                &\stackrel {(4.1)}=& (r'\o a')((r''\o a'')(r\o a)).
\end{eqnarray*}
\end{proof}

So $R\o A$ is an associative algebra with this product. We will use the notation $R\#A$ to denote $R\o A$ with the above product, and the elements  $x\o a$ will be denoted by $x\#a$.
We know that $R\#A$ can be considered as a $M(A\o A)$-module. So we can consider the space $E\rhd(R\#A)$. The element $E\rhd(r\#a)$ in $R\#A$ can be  denoted by $E_{(1)}\rhd r\# E_{(2)}a$.  The formula is well covered because $R$ is unital.

\begin{proposition}
For any $r,s \in R, a,b\in A$. We have
$$(E\rhd(r\# a))(s\#b)=(r\#a)(s\#b)=(r\# a)(E\rhd(s\#b)),$$
and
$$E\rhd((r\#a)(s\#b))=(E\rhd(r\# a))(s\#b).$$
\end{proposition}

\begin{proof}
For any $r,s \in R, a,b\in A$, we have
\begin{eqnarray*}
(E\rhd(r\# a))(s\#b) &=& (E_{(1)}\rhd r\# E_{(2)}a)(s\#b)\\
                    &\overset{(4.1)}{=}& \sum (E_{(1)}\rhd r)(E_{(2)}a)_{(1)}\rhd s\#(E_{(2)}a)_{(2)}b\\
                    &\overset{(2.1)}=& \sum  (E_{(1)}\rhd r)((E_{(2)}a_{(1)})\rhd s)\#a_{(2)}b\\
                    &\overset{(3.3)}{=}& \sum r(a_{(1)}\rhd s)\#a_{(2)}b\\
                    &\overset{(4.1)}{=}& (r\#a)(s\#b).
\end{eqnarray*}

On the other hand,
\begin{eqnarray*}
(r\#a)(E\rhd(s\#b)) &=&  (r\#a)(E_{(1)}\rhd s\# E_{(2)}b)\\
                    &=& \sum r(a_{(1)}E_{(1)}\rhd s)\#a_{(2)}E_{(2)}b\\
                    &\overset{(2.2)}=& \sum r(a_{(1)}\rhd s)\#a_{(2)}b\\
                    &=& (r\#a)(s\#b).
\end{eqnarray*}
 Now let us check the final equation.
\begin{eqnarray*}
E\rhd((r\#a)(s\#b)) &=& \sum E_{(1)}\rhd (r(a_{(1)}\rhd s))\#E_{(2)}a_{(2)}b\\
                &=& \sum (E_{(1)}\rhd r)(E_{(2)}a_{(1)}\rhd s)\# E_{(3)}a_{(2)}b\\
                &\overset{(2.1)}=& \sum (E_{(1)}\rhd r)(E_{(2)}a_{(1)}\rhd s)\#a_{(2)}b\\
                &=& (E\rhd(r\# a))(s\#b).
\end{eqnarray*}
This completes the proof.
\end{proof}

So $E\rhd(R\#A)$ is a subalgebra of $(R\#A)$, it is also a right ideal. Now we  will investigate the product in $E\rhd(R\#A)$.

\begin{proposition}\label{4.3}
The product in $E\rhd(R\#A)$ is non-degenerate.
\end{proposition}

\begin{proof}
For any $\sum r_{i}\# a_{i}$ in $R\#A$, assume that $(\sum r_{i}\# a_{i})(s\#b)=0$ for all $s\in R, b\in A$. Then we get
$$\sum r_{(i)}(a_{i(1)}\rhd s)\# a_{i(2)}b=0.$$
Apply $\D$ and multiply by $c\in A$, then we have
$$\sum r_{(i)}(a_{i(1)}\rhd s)\o a_{i(2)}b_{(1)}\o a_{i(3)}b_{2}c=0$$
where $b_{(2)}$ is covered by $c$. Since $\D(A)(1\o A)=E(A\o A))$ and $\D(a)E=\D(a)$, so
$$\sum r_{(i)}(a_{i(1)}\rhd s)\o a_{i(2)}p\o a_{i(3)}q=0$$
for all $s\in R, p,q\in A$. Apply $S$ and replace $s$ by $mc$ , this gives
$$\sum r_{(i)}((a_{i(1)}mc)\rhd s)\o S(a_{i(2)})b\o a_{i(3)}q=0$$
for all $m\in M(A),s\in R, b,c,q\in A$. Now replace $m$ by $S(a_{i(2)})b$ then we get
$$\sum r_{i}(a_{i(1)}S(a_{i(2)})bc\rhd s)\o a_{i(3)}q=0.$$
Hence
$$\sum (E_{(1)}\rhd r_i)(cb\rhd s)\o E_{(2)}a_iq=0$$
for all $c,b,q\in A$ and $s\in R$. Because $R$ is unital and $A$ is idempotent we can cancel $cb\rhd s$ and $q$, then we obtain
$$ E\rhd(\sum r_i\o a_i)=0.$$
On the other hand, suppose that $(s\#b)(\sum r_{i}\# a_{i})=0$ for all $s\in R$ and $b\in A$, then
$$\sum s(b_{(1)}r_i)\o b_{(2)}a_i=0.$$
Multiply $a'$ from the left and use the fact $(1\o A)\D(A)=(A\o A)E$, we get
$$s(bE_{(1)}\rhd r_i)\o b'E_{(2)}a_i=0.$$
As before, we can cancel $s,b,b'$ to get again $ E\rhd(\sum r_i\o a_i)=0.$
\end{proof}

{\bf{Notation.}} For the algebra $E\rhd (R\#A)$, we will use the notation $R\#_EA$ instead and the element
 $E\rhd(r\#a)$ will be denoted by $r\#_Ea$. And we call it the {\it smash product algebra}.\\

We can also form the smash product in the following way.  Let $A$ be a regular weak multiplier Hopf algebra and let $R$ be a left $A$-module algebra. Then we have the following lemma.

\begin{lemma}
 $R$ is a unital  right $\v_t(A)$-module via
$$r\cdot \v_t(a)=S^{-1}(\v_t(a))\rhd r= \v'_s(a)\rhd r$$
for all $r\in R, a\in A$.
\end{lemma}

\begin{proof}
 It is easy to know that $R$ is a right $\v _t(A)$-module. From Proposition 2.10 in \cite{VDW3}, $A$ can be regarded as a unital left $\v'_s(A)$-module.  Because $R$ is unital, we have $R=AR=\v'_s(A)AR=\v'_s(A)R.$
\end{proof}

Also by Proposition 2.10 in \cite{VDW3}, we know that $A$ is a unital left $\v_t(A)$-module. So we can define the tensor product $R\o_{\v_t(A)}A$. For  concisely we will denote it by $R\o_t A$.

\begin{proposition}
Let $R$ be a left $A$-module algebra. Then the  space $R\o_t A$  is a non-degenerate algebra with the product
$$(r\o_t a)(s\o_t b)=\sum r(a_{(1)}\rhd s)\o_t a_{(2)}b$$
for all $r,s\in R$ and $a,b\in A$.
\end{proposition}

\begin{proof}
Remark that $a_{(2)}$ is covered by $b$ or $s$ can be used to cover $a_{(1)}$. First we show that the multiplication above is well defined. It means we have to show
$$(rx\o_t a)(s\o_t b)=(r\o_t xa)(s\o b)$$
for all $r\in R, a\in A, x\in \v_t(A).$ We have
\begin{eqnarray*}
(rx\o_t a)(s\o_t b) &=& \sum (rx)(a_{(1)}\rhd s)\o_t a_{(2)}b\\
                    &=& \sum (S^{-1}(x)r)(a_{(1)}\rhd s)\o_t a_{(2)}b\\
                    &=& \sum r(xa_{(1)}\rhd s)\o_t a_{(2)}b\\
                    &=& (r\o_t xa)(s\o_t b).
\end{eqnarray*}
The proof of the  associativity of the product is straightforward. In a similar way as in Proposition \ref{4.3} we can get the non-degeneracy of the product.
\end{proof}

Now we show that there is an isomorphism between $R\o_t A$ and $R\#_EA$.

\begin{proposition} We have that
$R\o_t A$ is isomorphic with $R\#_EA$.
\end{proposition}

\begin{proof}
For any $r\in R, a\in A, x\in \v_t(A)$, define the map
$$f: R\o_t A\lr R\#_EA: r\o xa \mapsto E_{(1)}r\o E_{(2)}xa.$$
First we have to show that  $f$ is well-defined. It means that we must get
$$E_{(1)}r\o E_{(2)}xa=E_{(1)}S^{-1}(x)r\o E_{(2)}a.$$
This is true since $E(y\o 1)=E(1\o S(y))$ for any $y\in \v_s(A)$. Obviously $f$ is bijective. Finally it is easy to show that $f$ is an isomorphism.
\end{proof}

In this paper  we will mainly consider the algebra $R\#_EA$. Let us study more properties  about the algebra.

The product in $R\#_EA$ is defined by the twist map
$$T: A\o R\lr R\o A: a\o r\mapsto \sum a_{(1)}\rhd  r\o a_{(2)}$$
where $a\in A, r\in R$. We also know that $T$ is bijective and its inverse is given by
$$T^{-1}: R\o A\lr A\o R: x\o a\mapsto \sum a_{(2)}\o S_A^{-1}(a_{(1)})\rhd  r.$$
So the product is given by $(m_R\o m_A)(id\o T\o id)$, here $m_R$ and $m_A$ denote the multiplication in $R$ and $A$,
 respectively. Remark that for the last expression, $a_{(1)}$ is covered by $r$.

If $R$ has an identity $1_R$, then we have
\begin{eqnarray*}
(E\rhd(1_R\# a))(E\rhd(1_R\# b)) &=& E\rhd ((1_R\# a)(1_R\# b))\\
               &=& \sum \v_t(a_{(1)})\rhd 1_R\#_Ea_{(2)}b\\
               &=& \sum 1_R\#_E \v_t(a_{(1)})a_{(2)}b\\
               &=& 1_R\#_E ab.
\end{eqnarray*}
It means that $a\mapsto 1\#_Ea$ is a homomorphism of $A$ into $R\#_EA$.

If $A$ has an identity, then $A$ is a weak Hopf algebra with  $E=\D(1)=1_{1}\o1_2$(see Proposition 4.12 in \cite{VDW2}). And we have
$$(r\#1)(s\#1)=r(1_1\rhd  s)\#1_2=rs(S(1_1)1_2\rhd  1)\# 1=rs\# 1.$$
This equality also gives a homomorphism of $R$ into $R\#A$ by $r\mapsto r\#1$.

If $A$ is a regular multiplier Hopf algebra, then it is the case which has been studied in \cite{DVZ}. If $R$ and $A$ have identities, then it is  case which has been studied in \cite{N,ZW}. And $1\o_t 1$ is the identity in $R\# A$.

Now let us consider the following useful homomorphisms which are appeared in Proposition 5.7 in \cite{DVZ}. We will generalize it to weak multiplier Hopf algebras case.

\begin{proposition}
Let $A$ and $R$ be as before. For all $r,r'\in R, a,a'\in A$
\begin{romanlist}[(ii)]
\item Define  $\pi_A: A\lr M(R\#_EA)$  by
$$\pi_A(a)(r\#_Ea')=\sum a_{(1)}r\#_Ea_{(2)}a'$$
$$(r\#_Ea')\pi_A(a)=r\#_Eaa'.$$
Then $\pi_A$ is a unital algebra homomorphism.

\item Define $\pi_R: R\lr M(R\#_EA)$ by
$$\pi_R(r)(r'\#_Ea)=rr'\#_E a$$
$$(r'\#_Ea)\pi_R(r)=\sum r'(a_{(1)}\rhd r)\#_Ea'_{(2)}.$$
Then $\pi_R$ is an algebra homomorphism.
\end{romanlist}
\end{proposition}

\begin{proof}
(1) First we show that $\pi_A(a)\in M(R\#_EA)$. For all $r,r'\in R$ and $a,a',a''\in A$ we have
\begin{eqnarray*}
(r'\#_E a'')(\pi_A(a)(r\#_E a'))&=& (r'\#_E a'')(\sum a_{(1)}\rhd  r\#_Ea_{(2)}a')\\
                               &=& \sum r'((a''_{(1)}a_{(1)})\rhd r)\#_E a''_{(2)}a_{(2)}a'\\
                               &=& ((r'\#_E a'')\pi_A(a))(r\#_E a').
\end{eqnarray*}
So $\pi_A$ is well defined. Next we will show that  $\pi_A$ is a homomorphism. For any $a,a'a''\in A, r,r'\in R$, we have
\begin{eqnarray*}
\pi_A(aa')(r''\#_Ea'') &=& \sum (a_{(1)}a'_{(1)})\rhd r\#_Ea_{(2)}a'_{(2)}a''\\
                     &=& \sum (\pi_A(a)\pi_A(a'))(r''\#_Ea'')
\end{eqnarray*}
and
\begin{eqnarray*}
(r''\#_Ea'')\pi_A(aa')=r\#_E a''aa'=(r''\#_Ea'')(\pi_A(a)\pi_A(a')).
\end{eqnarray*}
Since $A$ is idempotent and $\D(A)(1\o A)=E(A\o A)$, we can get $\pi_A(A)(R\#_EA)=R\#_EA$. Then $\pi_A$ is a unital algebra homomorphism.

(2) The proof is similar.
\end{proof}

\begin{remark} With the above notation.
 $\pi_R$ is not unital any more. But if $R^2=R$, then $\pi_R$ is unital. For any $r,r'\in R$ and $a\in A$, we have
$$(r'\#_Ea)\pi_R(r)=\sum r'(a_{(1)}\rhd r)\#_Ea'_{(2)}.$$
Because $\D(A)(1\o A)=E(A\o A)$ and $R$ is unital, so the right-hand side is the element of the form
$$ r'(E_{(1)}\rhd r)\#_EE_{(2)}a.$$
It is also equal to
$$E_{(1)}(r'(\v_s(E_{(2)}))r)\# E_{(3)}a.$$
If $R^2=R$, then the element of  this form is actually equal to $r\#_Ea$.
\end{remark}

\begin{proposition}\label{4.8}
 For all $r\in R, a\in A$, we have
$$\pi_A(a)\pi_R(r)=\sum a_{(1)}\rhd r\#_E a_{(2)}$$
and
$$\pi_R(r)\pi_A(a)=r\#_Ea.$$
\end{proposition}

\begin{proof}
The proof is straightforward.
\end{proof}

It follows that $R\#_EA=\pi_R(R)\pi_A(A)$=$\pi_A(A)\pi_R(R)$, with this proposition we can extend $\pi_R$ from $R$ to $M(R)$ . We also have the following universal property.

\begin{proposition}
Let $A$ be a regular weak multiplier Hopf algebra, assume that $R$ is an $A$-module algebra. Let $B$ be an algebra over $\mathbb{C}$. If there exists homomorphisms $\pi_A: A\lr M(B)$ and $\pi_R: R\lr M(B)$ such that
$$\pi_A(a)\pi_R(r)=\sum \pi_R(a_{(1)}\rhd r)\pi_A(a_{(2)})$$
for all $a\in A, r\in R$. Then there is a homomorphism $\pi: R\#_E A\lr M(C)$ such that
$$\pi(r\#_E a)=\pi_R(r)\pi_A(a).$$
\end{proposition}

\begin{proof}
Straightforward.
\end{proof}

Now let us recall the definition of covariant module in \cite{DVZ}.

\begin{definition}
Let $A$ be a regular weak multiplier Hopf algebra and $R$ be a left $A$-module algebra. Let $V$ be a vector space which is both a left $A$-module and a left $R$-module. We say that $V$ is a covariant $A$-$R$-module if
$$
a(rv)=\sum (a_{(1)}r)(a_{(2)}v)
$$
for all $a\in A,r\in R, v\in V$.
\end{definition}

$V$ is called unital if it is unital both for $A$ and $R$. And $V$ is called non-degenerate if it is non-degenerate both for $A$ and $R$. Note that if $V$ is unital, then $V$ is automatically a non-degenerate $A$-module.

The following theorem gives  the relation between the covariant modules and the smash product modules (cf. \cite{WL}). The result is similar to the result for actions of locally compact groups on $C^\ast$-algebras(see Section 7.6 in \cite{P}). And in \cite{DVZ} the authors also obtain the result for multiplier Hopf algebras. Now we are going to show that it is also true for weak  multiplier Hopf algebras.

\begin{theorem}
Assume that $R^2=R$. Then there is a one-to-one correspondence between unital $R\#_EA$-modules and covariant $R$-$A$-modules with unital actions. And if one is non-degenerate, so is the other.
\end{theorem}

\begin{proof}
Suppose we have a left $R\#_EA$-module $V$ which is unital. Then we can extend $V$ to a $M(R\#_EA)$-module. With the homomorphisms $\pi_A$ and $\pi_R$, we have actions of $A$ and $R$ on $V$. With the formula in Proposition \ref{4.8} we have
\begin{eqnarray*}
a(rv)=a(\pi_R(r)v)=\pi_A(a)(\pi_R(r)v)=\sum \pi_R(a_{(1)}r)(\pi_A(a_{(2)})v)=\sum (a_{(1)}r)(a_{(2)}v)
\end{eqnarray*}
for any $a\in A, r\in R, v\in V$. Here $a_{(1)}$ is covered by $r$. So $V$ is a covariant $A$-$R$-module. Because
$$AV=\pi_A(A)(R\#_EA)V=(R\#EA)V=V$$
and
$$\quad  RV=\pi_R(R)(R\#_EA)V=(R\#EA)V=V,$$
we get $V$ is a unital covariant module.

Conversely, suppose that $V$ is a unital covariant $R$-$A$-module. For any $a\in A, r\in R, v\in V$, define the action of $R\#_EA$ on $V$ by
$$(r\#_Ea)v=r(av).$$
It is easy to show that $V$ is a left $R\#_EA$-module. And since $V$ is unital, we have
$$(R\#_EA)V=R(AV)=RV=V.$$
So $V$ is a unital left $R\#_EA$-module.

Now let us discuss the non-degeneracy of  these actions. Suppose that $V$ is a unital non-degenerate $R\#_EA$-module. Take any $v\in V$ and assume that $rv=0$ for all $r\in R$. Then $rav=0$ for all $r\in R, a\in A$. So we can get $v=0$. It means that $V$ is a non-degenerate $R$-module. Because $V$ is automatically a non-degenerate $A$-module, we get that $V$ is a unital non-degenerate covariant $R$-$A$-module.

Conversely, suppose that $V$ is a non-degenerate $R$-module. Take any $v\in V$ and assume that $(r\#_Ea)v=0$ for all $r\in R, a\in A$. Then $r(av)=0$ for all $r\in R, a\in A$. Because $V$ is a non-degenerate covariant  $R$-$A$-module, it follows that $v=0$. This means that $V$ is a non-degenerate $R\#_EA$-module.
\end{proof}

\section{The dual pairs}\label{five}
In this section we will consider the dual pair of regular weak multiplier Hopf algebras. For a dual pair $\<A,B\>$ of regular weak multiplier Hopf algebras, we can define the smash product $A\#_{E_B}B$ and $B\#_{E_A}A$. For the algebra $B\#_{E_A}A$ we have the following faithful action.

\begin{proposition}
$B$ is a left $(B\#_{E_A}A)$-module by
\begin{equation}
(b\#_Ea)\cdot b'=b(a\rhd b')=\sum bb'_{(1)}\<a, b'_{(2)}\>
\end{equation}
for any $a\in A,b,b'\in B$. And the action is faithful.
\end{proposition}

\begin{proof}
For all $a',a''\in A$ and $b,b',b''\in B$, we have
\begin{eqnarray*}
((b''\#_Ea'')(b'\#_Ea'))\cdot b &\overset{(4.1)}{=}& (b''(a''_{(1)}\rhd b')\#_Ea'_{(2)}a')\cdot b\\
                            &\overset{(5.1)}{=}& b''(a''_{(1)}\rhd b')(a'_{(2)}a'\rhd b )\\
                            &\overset{(3.1)}{=}& b''(a''\rhd (b'(a'\rhd b')))\\
                            &\overset{(5.1)}{=}& (b''\#_Ea'')\cdot (b'(a'\rhd b))\\
                            &=& (b''\#_Ea'')(b'\#_Ea')\cdot b.
\end{eqnarray*}
So $B$ is a left $(B\#A)$-module. Here $a''_{(1)}$ is covered by $b'$.  Next we show that the action is faithful. Let $\sum b_i\#_Ea_i\in B\#_EA$ and assume that $\sum (b_i\#_Ea_i)\cdot b'=0$ for all $b'\in B$. Then
$$\sum \<a_i, b'_{(2)}\>b_ib'_{(1)}=0.$$
Multiply by $b''$ on the right we get that
$$\sum b_i\<a_i, E_{B(2)}p\>E_{B(1)}q=0$$
for any $p,q\in B$. So
$$\sum b_iE_{B(1)}\o a_i\lhd E_{B(2)}=0.$$
Now we claim that
$$bE_{B(1)}\o a\lhd E_{B(2)}=E_{A(1)}\rhd b\o E_{A(2)}a$$
for all $a\in A, b\in B$. And if this is true, we finish the proof. Indeed, for any $a'\in A, b'\in B$,
\begin{eqnarray*}
\<a',bE_{B(1)}\>\<a\lhd E_{B(2)}, b'\>
 &=& \sum \<\<a'_{(1)}, b\>a'_{(2)},E_{B(1)}\>\<a, E_{B(2)}b'\>\\
 &=& \sum \<a'_{(1)},b\>\v_A(a'_{(2)}a_{(1)})\<a_{(2)},b'\>\\
 &\overset{(2.2)} =& \sum \<a'_{(1)}E_{(1)},b\>\v_A(a'_{(2)}E_{(2)}a_{(1)})\<a_{(2)},b'\>\\
 &\overset{(2.8)} =& \sum \<a'_{(1)}a_{(2)}S^{-1}(a_{(1)}), b\>\v_A(a'_{(2)}a_{(3)})\<a_{(4)},b'\>\\
  &=& \sum \<a'a_{(2)}S^{-1}(a_{(1)}), b\>\<a_{(3)}, b\>\\
   &\overset{(2.8)} =& \<a', E_{A(1)}\rhd b\>\<E_{A(2)}a, b'\>.
\end{eqnarray*}
In the second equality and the fifth equality,  we use
$$\<a\o a',E\>=\<a\o a',\D(1)\>=\<aa',1\>=\v(aa').$$
 Note that everything is well covered here. For example, in the second equality   $a'_{(2)}$ is covered by $\sum a_{(1)}\<a_{(2)},b'\>\in A$.
\end{proof}

\begin{proposition}
$B\#_{E_A}A$ can be considered as the span of the elements $ab$ in the algebra $C$ generated by $A$ and $B$ subject to the following commutation relation
$$ab=\sum \<a_{(1)},b_{(2)}\>b_{(1)}a_{(2)}$$
for all $a\in A, b\in B$.
\end{proposition}

\begin{proof}
From the definition of pairing we can get
$$\sum \<a_{(1)},b_{(2)}\>a_{(2)}\o b_{(1)}\in A\o B.$$
Then $C$ is well-defined. Define
$$\phi: B\#_EA\lr C,\quad b\#_Ea \mapsto ba$$
for $a\in A, b\in B$.   Since $B$ is a faithful module and we also have an action of $C$ on $B$, then $\phi$ is an injective homomorphism.
\end{proof}

Similarly $A\#_E B$ is the span of the elements $ab$ in the algebra generated by $A$ and $B$ subject to the commutation relation
$$ba=\sum \<a_{(2)},b_{(1)}\>a_{(1)}b_{(2)}$$
for any $a\in A, b\in B$. Then we can get the following easy result.

\begin{proposition} We have that
$B\#_{E_A}A$ is anti-isomorphic with $A\#_{E_B}B$.
\end{proposition}

The anti-isomorphism is defined as $b\#_{E_A}a \mapsto S^{-1}a\#_{E_B}S(b)$. With this map we can get a right  $(B\#_{E_A}A)$-module structure on $A$.

\begin{proposition}
For all $a,a'\in A, b\in B$, define
$$a'\lhd (b\#a)=(S^{-1}a\#_{E_B}S(b))\rhd a',$$
with this action, $A$ is a right faithful $(B\#_{E_A}A)$-module.
\end{proposition}

Inspired by the above module structure, $A$ can be also regarded as a faithful right $(B\#_{E_A}A)$-module with action
$$a'\lhd (b\#a)=(a'\lhd b)a$$
for any $a,a'\in A, b\in B$.

Now assume that $A$ is an algebraic quantum groupoid and $\hat{A}$ denotes the dual of $A$.  Recall that an algebraic quantum groupoid is a regular weak multiplier Hopf algebra with a faithful set of left integrals. Consider the dual pair $\<A,\hat{A}\>$, we have the smash product $A\#_{\hat{E}}\hat{A}$.
 And $A$ is a left faithful $A\#_{\hat{E}}\hat{A}$-module with the action given by
$$(a\#_{\hat{E}}b)a'=a(b\rhd a')=\sum aa'_{(1)}\<a'_{(2)}, b\>$$
for all $a,a'\in A, b\in \hat{A}$.

\begin{proposition}\label{5.5}
The algebra $A\#_{\hat{E}}\hat{A}$ as acting on $A$ is the span of the maps of the form $a'\mapsto aS(E_{(1)})\vp(E_{(2)}a')$ from $A$ to $A$, where $a,a'\in A$ and $\vp$ is a left integral on $A$.
\end{proposition}

\begin{proof}
Let $a'\in A,b=\vp(a'\cdot)\in \hat{A}$. By Remark 3.3, we have
\begin{eqnarray*}
(a\#b)\rhd a'' &=& \sum aa''_{(1)}\vp(a'a''_{(2)})\\
               &=& \sum aS(a'_{(1)})\vp(a'_{(2)}a'').
\end{eqnarray*}
Here we use the property of the integral. Since $S$ is bijective and $\D(A)(A\o 1)=E(A\o A)$, so
$$(a\#b)\rhd a''=pS(E_{(1)})\vp(E_{(2)}q)$$
where $p,q\in A$. This finishes the proof.
\end{proof}

Consider the algebra $A\o \hat{A}$ with the product given by
$$(a\o b)(a'\o b')= \<a',b\>a\o b'$$
where $a,a'\in A, b,b'\in \hat{A}$. We will use $A\diamondsuit \hat{A}$ to denote $A\o \hat{A}$ with above product
 and $a\diamondsuit b$ to denote the element $a\o b$. Let
$$A\bar{\diamondsuit}\hat{A}=\{S(E_{(1)}a)\o \vp(E_{(2)}c\cdot )\mid  a,c\in A\}.$$
It is a sub-algebra of $A\diamondsuit \hat{A}$. Following the  idea as in Proposition 6.7 in \cite{DVZ}, we have the following result.

\begin{proposition}\label{5.6} We have that
$A\#_{\hat{E}}\hat{A}$ is isomorphic with $A\bar{\diamondsuit}\hat{A}$.
\end{proposition}
\begin{proof}
We give the sketch of the proof. The isomorphism map $\theta: A\#_{\hat{E}}\hat{E}\lr  A\bar{\diamondsuit}\hat{A}$ is given by
$$\theta(a\#_{\hat{E}}\vp(c\cdot))=\sum aS(c_{(1)})\diamondsuit \vp(c_{(2)}\cdot).$$
And we know that $A$ is a faithful left  $A\bar{\diamondsuit} \hat{A}$-module with the action $(a\diamondsuit b)\rhd a'=\<a',b\>a$ for $a,a'\in A, b\in \hat{A}$.
 By Proposition \ref{5.5}, we can get that $\theta$ is an isomorphism.
\end{proof}

\section{Duality Theorem}\label{six}

In this section we will study the bi-smash products and obtain the duality theorem for an algebraic quantum groupoid $A$.  Let $(A,\D,E_A)$ be an algebraic quantum groupoid, acting on an algebra $R$, so we can define the smash product $R\#_EA$. Let $B$ be another  algebraic quantum groupoid paired with $A$. Also assume the pairing is non-degenerate. Given a pairing $\<A, B\>$, there is a left action  of $B$ on $A$ $: b\rhd a=\sum a_{(1)}\<a_{(2)}, b\>$, with this action we can give $R\#_EA$ a $B$-module algebra structure.

\begin{proposition}
$R\#_EA$ is a left $B$-module algebra with the action $\cdot$ given by
$$b\cdot (r\#_E a)= r\#_E(b\rhd a)=\sum r\#_Ea_{(1)}\<a_{(2)}, b\>$$
for all $a\in A, b\in B, r\in R$.
\end{proposition}

\begin{proof}
First we show it is a unital module. Let $b,b'\in B, r\in R, a\in A$
\begin{eqnarray*}
(bb')\cdot (r\#_Ea) &=& r\#_E(bb'\rhd a)\\
                    &=& \sum r\#_Ea_{(1)}\<a_{(2)}, bb'\>\\
                    &=& \sum r\#_Ea_{(1)}\<a_{(2)}, b\>\<a_{(3)}, b'\>\\
                    &=& \sum b\cdot (b'\cdot (r\#_Ea)).
\end{eqnarray*}
Since $A$ is a unital $B$-module, so $R\#_EA$ is unital. For $b\in B, r,r'\in R, a,a'\in A$, we have
\begin{eqnarray*}
b\cdot ((r\#_Ea)(r'\#_Ea')) &=& \sum b\cdot (ra_{(1)}r'\#_Ea_{(2)}a') \\
                            &=& \sum ra_{(1)}r'\#_E (b\rhd (a_{(2)}a'))\\
                            &=& \sum ra_{(1)}r'\#_E a_{(2)}a'_{(1)}\<a_{(3)}, b_{(1)}\>\<a'_{(2)}, b_{(2)}\>\\
                            &=& \sum (r\#_E a_{(1)}\<a_{(2)},b_{(1)}\>)(r'\#_E a'_{(1)}\<a'_{(2)},b_{(2)}\>)\\
                            &=& \sum (b_{(1)}\cdot (r\#_E a))(b_{(2)}\cdot (r'\#_Ea')).
\end{eqnarray*}
So $R\#_EA$ is a $B$-module algebra. Remark that in the above calculations we always have the well coverings.
\end{proof}

Now we can form the bi-smash product $(R\#_{E_A}A)\#_{E_B}B$. And the following result is a consequence of
 Proposition 5.1.

\begin{proposition}
$R\#_{E_A}A$ is a faithful left $(R\#_{E_A}A)\#_{E_B}B$-module with  the action given by
\begin{eqnarray*}
((r\#_{E_A}a)\#_{E_B}b)(r'\#_{E_A}a') &=& (r\#_{E_A}a)(r'\#_{E_A}b\rhd a').
\end{eqnarray*}
\end{proposition}

Let $R$ be an $A$-module algebra.  Set
$$R\bar{\o}A=\{r\o a\in R\o A \mid r\o a=S^{-1}(E_{(1)})\cdot r\o aE_{(2)})\}$$
Note that $(1\o a)E$ belongs to $\v_s(A)\o A$. Now consider the  map $T: R \bar{\o} A\lr R\#_EA$ which is defined by
$$T(r \bar{\o} a)=\sum a_{(1)}r\#_Ea_{(2)}$$
for all $r\in R, a\in A$. In fact the map is bijective, and its inverse is given by
$$T^{-1}(r\#_Ea)=\sum S^{-1}(a_{(1)})r \bar{\o} a_{(2)}.$$
Remark that  $T$ is well-defined. For $a\in A, r\in R$, we get
\begin{eqnarray*}
T(S^{-1}(E_{(1)})\cdot r\o a E_{(2)}))     &=& T(\v'_t(a_{(1)})\cdot r\o a_{(2)})\\
                                    &=& a_{(2)}\v'_t(a_{(1)})\cdot r\#_E a_{(3)}\\
                                    &=& a_{(1)}\cdot r\#_E a_{(2)} \\
                                    &=& T(r\o a).
\end{eqnarray*}

And with the map $T$ we can define a new faithful action $\pi$ of $(R\#_{E_A}A)\#_{E_B}B$ on $R\bar{\o}A$.

\begin{proposition}
For all $a,a'\in A, r, r'\in R, b\in B$, define the action $\pi$ as
$$\pi((r\#_{E_A}a)\#_{E_B}b)(r'\bar{\o}a')=T^{-1}((r\#_{E_A}a)\#_{E_B}b)T(r'\bar{\o}a'),$$
with the action $\pi$, we have that $R\bar{\o}A$ is a faithful $(R\#_{E_A}A)\#_{E_B}B$-module.
\end{proposition}

\begin{proof}
First we have
\begin{eqnarray*}
T^{-1}bT(r\bar{\o}a) &=& T^{-1}\sum a_{(1)}\cdot r\#_E(b\rhd a_{(2)})\\
                     &=& \sum \<a_{(4)}, b\>S^{-1}(a_{(2)})a_{(1)}\cdot r\bar{\o}a_{(3)}\\
                     &=& \sum r\bar{\o}a_{(1)}\<a_{(2)}, b\>
\end{eqnarray*}
and
\begin{eqnarray*}
T^{-1}(r\#_Ea)T(r'\bar{\o} a') &=& T^{-1}\sum r(a_{(1)}a'_{(1)\cdot r'})\#_Ea_{(2)}a'_{(2)}\\
                               &=& \sum (S^{-1}(a_{(2)}a'_{(2)})\cdot r)(\v'_t(a_{(1)}a'_{(1)})\cdot r')\bar{\o}a_{(3)}a'_{(3)}\\
                               &=& \sum(S^{-1}(a_{(1)}a'_{(1)})\cdot r)r'\bar{\o}a_{(2)}a'_{(2)}.
\end{eqnarray*}
Then we get
$$\pi((r\#_{E_A}a)\#_{E_B}b)(r'\bar{\o}a')=\sum (S^{-1}(a_{(1)}a'_{(1)})\cdot r)r'\bar{\o}a_{(2)}a'_{(2)}\<a'_{(3)}, b\>.$$
Next we are going to show that the map gives a module action.
\begin{eqnarray*}
&& \pi[((r\#_{E_A}a)\#_{E_B}b)((r'\#_{E_A}a')\#_{E_B}b')](r''\bar{\o}a'')\\
&=& \sum \pi(\<a'_{(2)},b_{(1)}\>(x(a_{(1)}x')\#_{E_A}a_{(2)}a'_{(1)})\#_{E_B}b_{(2)}b')(r''\bar{\o}a'')\\
&=& \sum \<a'_{(3)}a''_{(3)}, b\>\<a''_{(4)}, b'\>(S^{-1}(a_{(2)}a'_{(1)}a''_{(1)})(x(a_{(1)}x')))x''\bar{\o}a_{(3)}a'_{(2)}a''_{(2)}\\
&=& \sum  \<a'_{(4)}a''_{(4)}, b\>\<a''_{(5)}, b'\> (S^{-1}(a_{(3)}a'_{(2)}a''_{(2)})x)(S^{-1}(a_{(2)}a'_{(1)}a''_{(1)})((a_{(1)}x')))x''\bar{\o}a_{(4)}a'_{(3)}a''_{(3)}\\
&=& \sum  \<a'_{(4)}a''_{(4)}, b\>\<a''_{(5)}, b'\>(S^{-1}(a_{(1)}a'_{(2)}a''_{(2)})x)(S^{-1}(a'_{(1)}a''_{(1)})x')x''\bar{\o}a_{(2)}a'_{(3)}a''_{(3)}
\end{eqnarray*}
For the last equality we use $r\o a=S^{-1}(E_{(1)})\cdot r\o aE_{(2)})$. And remark that we have well coverings in the above equalities. Such as in the first equality $a_{(1)}$ is covered by $x'$, and then the others are well covered. On the other side we have
\begin{eqnarray*}
&& \pi [(r\#_{E_A}a)\#_{E_B}b]\pi [(r'\#_{E_A}a')\#_{E_B}b'](r''\bar{\o}a'')\\
&=& \pi(((r\#_{E_A}a)\#_{E_B}b))\sum  \<a''_{(3)}, b'\>(S^{-1}(a'_{(1)}a''_{(1)})x')x''\bar{\o}a'_{(2)}a''_{(2)}\\
&=&\<a'_{(4)}a''_{(4)},b\>\<a''_{(5)}, b'\>(S^{-1}(a_{(1)}a'_{(2)}a''_{(2)})x)(S^{-1}(a'_{(1)}a''_{(1)})x')x''
\bar{\o}a_{(2)}a'_{(3)}a''_{(3)}.
\end{eqnarray*}
So we get a left $(R\#_{E_A}A)\#_{E_B}B$-module $R\bar{\o}A$. Since $T$ is bijective and combine with  Proposition 6.2,
 we have that $R\bar{\o}A$ is a faithful $(R\#_{E_A}A)\#_{E_B}B$-module.
\end{proof}

Now let  us specialize to the case of an algebraic quantum groupoid $A$ with a left integral $\vp$. Set
$$A\bar{\o}\hat{A}=\{a\o \vp(c\cdot)\in A\diamondsuit \hat{A}\mid c\o a=E_{(1)}c\o E_{(2)}a, \forall a,c\in A \}.$$
Following the same idea of Proposition \ref{5.6} or Theorem 7.6 in \cite{DVZ}, we get the following duality theorem.

\begin{theorem}\label{6.4} Let $A$ be an algebraic quantum groupoid with a left integral $\vp$ and with the dual $\hat{A}$.
  Let $R$ be a left $A$-module algebra. Then we have that
$(R\#_{E_A}A)\#_{E_{\hat{A}}}\hat{A}\cong R\o(A\bar{\o} \hat{A})$ as algebras.
\end{theorem}

\begin{proof}
For $a,a'\in A, b\in \hat{A}, r,r'\in R$, the resulting of the algebra $(R\#_{E_A}A)\#_{E_{\hat{A}}}\hat{A}$ acting on $R\bar{\o}A$
 is the span of operators of the form
$$r'\bar{\o}a' \mapsto  \sum (S^{-1}(a_{(1)}a'_{(1)})\cdot r)r'\bar{\o}a_{(2)}a'_{(2)}\<a'_{(3)}, b\>.$$
We denote the algebra of operators by  $P$. Since $T$ is bijective and  write $b$ as $\vp(c\cdot)$ where $c\in A$, we find that $P$ is spanned by the operators of the form
$$r'\bar{\o}a' \mapsto \sum \vp(c_{(3)}a')(c_{(2)}x)x'\bar{\o} aS(c_{(1)}).$$
We can replace $\sum c_{(2)}\o aS(c_{(1)})$ by $E_{(1)}p\o E_{(2)}q$, for any $p, q\in A$
 and with the fact that $T$ is bijective then $P$ is spanned by the operators of the form
$$r'\bar{\o}a' \mapsto rr'\bar{\o}\vp(E_{(1)}pa')E_{(2)}q.$$
Since both modules are faithful, then we  get an isomorphism.
\end{proof}

 Let $H$ be a regular  weak multiplier Hopf algebra with an identity. Then it is a a weak Hopf algebra (see Proposition 4.12 in \cite{VDW2}).
 In this case we obtain the following duality theorem for actions of quantum groupoids which was proven in \cite{N}.

\begin{corollary}[\cite{N} Lemma 3.1]
Let $H$ be a finite-dimensional  weak Hopf algebra and $A$ be a left $H$-module algebra. Then we can form the smash product $A\#H$ and $(A\#H)\#\hat{H}$. The map $\a: (A\#H)\#\hat{H}\lr End(A\#H)_A$ defined by
$$\a((x\#h)\#\phi)(y\#g)=(x\#h)(y\#(\phi\rightharpoonup g))=x(h_{(1)}\cdot y)\#h_{(2)}(\phi\rightharpoonup g)$$
for all $x,y\in A, h,g\in H, \phi\in \hat{H}$ is an isomorphism of algebras.
\end{corollary}

In \cite{N} it is proved by a straightforward computation. Remark that  we use a different method to prove the theorem. And also we generalize it to the infinite case.

If the canonical idempotent $E$ is equal to $1$, then the regular  weak multiplier Hopf algebra $A$ is a regular multiplier Hopf algebra.
 And now we can get  the main result in \cite{DVZ}.

\begin{corollary}[\cite{DVZ} Theorem 7.6]
If $A$ is an algebraic quantum group, acting on an algebra $R$ and if $\hat{A}$ is the dual of $A$, acting on the smash product $R\#A$ by means of the dual action, then the bi-smash product $(R\#A)\#\hat{A}$ is isomorphic with $R\o (A\diamondsuit \hat{A})$.
\end{corollary}

\section*{Acknowledgments}

The authors are very thankful to Professor A. Van Daele for his valuable comments on this paper.
The authors are also very grateful to the anonymous referee for his/her thorough review of this work and his/her comments and suggestions
 which help to improve the first and the second version of this paper.  The work was partially supported by the NSF of
 China (No. 11371088) and the NSF of China (No.11571173).


\begin{thebibliography}{0}

\bibitem{BM}  R. Blattner, S. Montgomery, A duality theorem for Hopf module algebras, {\it J. Algebra}. {\bf 95}(1) (1985) 153--172.

\bibitem{B1} G. B\"{o}hm, Comodules over weak multiplier bialgebras, {\it Internat. J. Math.} {\bf 25}(5) (2014) 1450037.

\bibitem{B2} G. B\"{o}hm, Yetter-Drinfeld modules over weak multiplier bialgebras, {\it Israel J. Math.} {\bf 209}(1) (2015) 85--123.

\bibitem{BGL} G. B\"{o}hm, J. G\'{o}mez-Torrecillas and E. L\'{o}pez-Centella, Weak multiplier bialgebras,
 {\it Trans. Amer. Math. Soc.} {\bf 367}(12) (2015) 8681--8721.

\bibitem{BNS}  G. B\"ohm, F. Nill and K. Szlach\'{a}nyi, Weak Hopf algebras I. Integral theory and $C^*$-structure, {\it J. Algebra.} {\bf 221} (1999) 385--438.

\bibitem{BS}  G. B\"{o}hm, K. Szlach\'{a}nyi, Weak Hopf algeras II. Representation theory and
 the Markov trace, {\it J. Algebra.} {\bf 233} (2000) 156--212.

\bibitem{DV} B. Drabant, A. Van Daele, Pairing and quantum double of multiplier Hopf algebras, {\it Algebra and
 Representation Theory}. {\bf 4} (2001) 109--132.

\bibitem{DVZ} B. Drabant, A. Van Daele and Y.H.Zhang, Actions of multiplier Hopf algebras, {\it Comm. Algebra}. {\bf 27}(9) (1999) 4117--4172.

\bibitem{KV} B. J. Kahng, A. Van Daele, The Larson-Sweedler theorem for weak multiplier Hopf algebras, to appear in {\it Comm. Algebra}.
 arXiv:1406.0299v1 [math.RA].


\bibitem{N} D. Nikshych, A duality theorem for quantum groupoids, {\it Contemp. Math}. {\bf 267} (2000) 237--243.

\bibitem{PF} A. Paques, D. Fl\^{o}res, Duality for groupoid (co)actions, {\it Comm. Algebra}. {\bf 42} (2014) 637--663.

\bibitem{P} G. K. Pedersen, {\it $C^\ast$-algebras and their automorphism groups},  Academic Press, 1979.


\bibitem{R} A. Ramsy, Virtual groups and group actions, {\it Adv. Math.} {\bf 6} (1971) 253--322.

\bibitem{S} M. Sweedler, {\it Hopf Algebras}, Benjamin, New-York, (1969).

\bibitem{T} T. Timmermann,  Integration on algebraic quantum groupoids, {\it Internat. J. Math.} {\bf 27}(2)(2016):1650014.


\bibitem{VD1}  A. Van Daele, Multiplier Hopf algebras, {\it Trans. Amer. Math. Soc.} {\bf 342}(2) (1994) 917--932.

\bibitem{VD2} A. Van Daele,  An algebraic framework for group duality. {\it Adv. Math.} {\bf 140} (1998), 323--366.

\bibitem{VD3} A. Van Daele,  Separability idempotent and multiplier algebras, arXiv:1301.4398v2\break [math.RA].

\bibitem{VDW1}  A. Van Daele, S.H.Wang,  Weak multiplier Hopf algebras. Preliminaries, motivation and basic examples,
in Operator algebras and quantum groups, Vol. 98(5) (Banach Center Publication, Warsaw, 2012), pp.~367--415.

\bibitem{VDW2}  A. Van Daele, S. H. Wang, Weak multiplier Hopf algebras I. The main theory,
 {\it J. Reine Angew. Math}. {\bf 705} (2015) 155--209.

\bibitem{VDW3}  A. Van Daele, S.H.Wang, Weak multiplier Hopf algebras II. The source and target algebras,
    arXiv:1403.7906v2 [math.RA].

\bibitem{VDW4} A. Van Daele, S.H.Wang, Weak multiplier Hopf algebras III. Integrals and duality.
 Preprint in University of Leuven and Southeast University of Nanjing.

\bibitem{WL} S.H.Wang, J. Q. Li, On the twisted smash product for bimodule algebras and Drinfel'd double.
  {\it Comm. Algebra} {\bf 26} (1998), 2435-2444.


\bibitem{ZW} X. Zhou, S.H. Wang, The duality theorem for weak Hopf algebra (co)actions, {\it Comm. Algebra}. {\bf 38} (2010) 4613--4632.

\end{thebibliography}
\end{document}